\documentclass[12pt,reqno,draft]{amsart}
\usepackage{amssymb,amscd,url}

\begin{document}

\hyphenation{non-archi-me-dean}
\newif\ifdraft \drafttrue
\draftfalse
\newcommand{\DRAFTNUMBER}{4}
\newcommand{\DATE}{July 4, 2007}
\newcommand{\TITLE}{Nonarchimedean Green Functions and
Dynamics on Projective Space}
\newcommand{\TITLERUNNING}{Nonarchimedean Green Functions and Dynamics}



\newtheorem{theorem}{Theorem}
\newtheorem{lemma}[theorem]{Lemma}
\newtheorem{conjecture}[theorem]{Conjecture}
\newtheorem{question}[theorem]{Question}
\newtheorem{proposition}[theorem]{Proposition}
\newtheorem{corollary}[theorem]{Corollary}
\newtheorem*{claim}{Claim}
\theoremstyle{definition}
\newtheorem*{definition}{Definition}
\newtheorem{example}[theorem]{Example}

\theoremstyle{remark}
\newtheorem{remark}[theorem]{Remark}
\newtheorem*{acknowledgement}{Acknowledgements}



\newenvironment{notation}[0]{%
  \begin{list}%
    {}%
    {\setlength{\itemindent}{0pt}
     \setlength{\labelwidth}{4\parindent}
     \setlength{\labelsep}{\parindent}
     \setlength{\leftmargin}{5\parindent}
     \setlength{\itemsep}{0pt}
     }%
   }%
  {\end{list}}

\newenvironment{parts}[0]{%
  \begin{list}{}%
    {\setlength{\itemindent}{0pt}
     \setlength{\labelwidth}{1.5\parindent}
     \setlength{\labelsep}{.5\parindent}
     \setlength{\leftmargin}{2\parindent}
     \setlength{\itemsep}{0pt}
     }%
   }%
  {\end{list}}
\newcommand{\Part}[1]{\item[\upshape#1]}

\renewcommand{\a}{\alpha}
\renewcommand{\b}{\beta}
\newcommand{\g}{\gamma}
\renewcommand{\d}{\delta}
\newcommand{\e}{\epsilon}
\newcommand{\f}{\varphi}
\newcommand{\bfphi}{{\boldsymbol{\f}}}
\renewcommand{\l}{\lambda}
\renewcommand{\k}{\kappa}
\newcommand{\lhat}{\hat\lambda}
\newcommand{\m}{\mu}
\newcommand{\bfmu}{{\boldsymbol{\mu}}}
\renewcommand{\o}{\omega}
\renewcommand{\r}{\rho}
\newcommand{\rbar}{{\bar\rho}}
\newcommand{\s}{\sigma}
\newcommand{\sbar}{{\bar\sigma}}

\renewcommand{\t}{\tau}
\newcommand{\z}{\zeta}

\newcommand{\D}{\Delta}
\newcommand{\G}{\Gamma}
\newcommand{\F}{\Phi}
\renewcommand{\L}{\Lambda}
\newcommand{\ga}{{\mathfrak{a}}}
\newcommand{\gb}{{\mathfrak{b}}}
\newcommand{\gn}{{\mathfrak{n}}}
\newcommand{\gp}{{\mathfrak{p}}}
\newcommand{\gP}{{\mathfrak{P}}}
\newcommand{\gq}{{\mathfrak{q}}}

\newcommand{\Abar}{{\bar A}}
\newcommand{\Ebar}{{\bar E}}
\newcommand{\Kbar}{{\bar K}}
\newcommand{\Pbar}{{\bar P}}
\newcommand{\Sbar}{{\bar S}}
\newcommand{\Tbar}{{\bar T}}
\newcommand{\ybar}{{\bar y}}
\newcommand{\phibar}{{\bar\f}}

\newcommand{\Acal}{{\mathcal A}}
\newcommand{\Bcal}{{\mathcal B}}
\newcommand{\Ccal}{{\mathcal C}}
\newcommand{\Dcal}{{\mathcal D}}
\newcommand{\Ecal}{{\mathcal E}}
\newcommand{\Fcal}{{\mathcal F}}
\newcommand{\Gcal}{{\mathcal G}}
\newcommand{\Hcal}{{\mathcal H}}
\newcommand{\Ical}{{\mathcal I}}
\newcommand{\Jcal}{{\mathcal J}}
\newcommand{\Kcal}{{\mathcal K}}
\newcommand{\Lcal}{{\mathcal L}}
\newcommand{\Mcal}{{\mathcal M}}
\newcommand{\Ncal}{{\mathcal N}}
\newcommand{\Ocal}{{\mathcal O}}
\newcommand{\Pcal}{{\mathcal P}}
\newcommand{\Qcal}{{\mathcal Q}}
\newcommand{\Rcal}{{\mathcal R}}
\newcommand{\Scal}{{\mathcal S}}
\newcommand{\Tcal}{{\mathcal T}}
\newcommand{\Ucal}{{\mathcal U}}
\newcommand{\Vcal}{{\mathcal V}}
\newcommand{\Wcal}{{\mathcal W}}
\newcommand{\Xcal}{{\mathcal X}}
\newcommand{\Ycal}{{\mathcal Y}}
\newcommand{\Zcal}{{\mathcal Z}}

\renewcommand{\AA}{\mathbb{A}}
\newcommand{\BB}{\mathbb{B}}
\newcommand{\CC}{\mathbb{C}}
\newcommand{\FF}{\mathbb{F}}
\newcommand{\GG}{\mathbb{G}}
\newcommand{\PP}{\mathbb{P}}
\newcommand{\QQ}{\mathbb{Q}}
\newcommand{\RR}{\mathbb{R}}
\newcommand{\ZZ}{\mathbb{Z}}

\newcommand{\bfa}{{\mathbf a}}
\newcommand{\bfb}{{\mathbf b}}
\newcommand{\bfc}{{\mathbf c}}
\newcommand{\bfe}{{\mathbf e}}
\newcommand{\bff}{{\mathbf f}}
\newcommand{\bfg}{{\mathbf g}}
\newcommand{\bfp}{{\mathbf p}}
\newcommand{\bfr}{{\mathbf r}}
\newcommand{\bfs}{{\mathbf s}}
\newcommand{\bft}{{\mathbf t}}
\newcommand{\bfu}{{\mathbf u}}
\newcommand{\bfv}{{\mathbf v}}
\newcommand{\bfw}{{\mathbf w}}
\newcommand{\bfx}{{\mathbf x}}
\newcommand{\bfy}{{\mathbf y}}
\newcommand{\bfz}{{\mathbf z}}
\newcommand{\bfA}{{\mathbf A}}
\newcommand{\bfF}{{\mathbf F}}
\newcommand{\bfB}{{\mathbf B}}
\newcommand{\bfG}{{\mathbf G}}
\newcommand{\bfI}{{\mathbf I}}
\newcommand{\bfM}{{\mathbf M}}
\newcommand{\bfzero}{{\boldsymbol{0}}}

\newcommand{\Aut}{\operatorname{Aut}}
\newcommand{\bad}{\textup{bad}}
\newcommand{\Disc}{\operatorname{Disc}}
\newcommand{\dist}{\Delta}  
\newcommand{\Div}{\operatorname{Div}}
\newcommand{\End}{\operatorname{End}}
\newcommand{\Family}{{\mathcal A}}  
\newcommand{\Fatou}{{\mathcal F}}
\newcommand{\Fbar}{{\bar{F}}}
\newcommand{\Gal}{\operatorname{Gal}}
\newcommand{\GL}{\operatorname{GL}}
\newcommand{\good}{\textup{good}}
\newcommand{\Index}{\operatorname{Index}}
\newcommand{\Image}{\operatorname{Image}}
\newcommand{\Julia}{{\mathcal J}}
\newcommand{\liftable}{{\textup{liftable}}}
\newcommand{\hhat}{{\hat h}}
\newcommand{\Ker}{{\operatorname{ker}}}
\newcommand{\Lift}{\operatorname{Lift}}
\newcommand{\mG}{\hat{g}}        
\newcommand{\MOD}[1]{~(\textup{mod}~#1)}
\newcommand{\Norm}{{\operatorname{\mathsf{N}}}}
\newcommand{\notdivide}{\nmid}
\newcommand{\normalsubgroup}{\triangleleft}
\newcommand{\odd}{{\operatorname{odd}}}
\newcommand{\onto}{\twoheadrightarrow}
\newcommand{\ord}{\operatorname{ord}}
\newcommand{\orbital}{\textup{orb-gd}}  
\newcommand{\Pic}{\operatorname{Pic}}
\newcommand{\Prob}{\operatorname{Prob}}
\newcommand{\Qbar}{{\bar{\QQ}}}
\newcommand{\rank}{\operatorname{rank}}
\newcommand{\Res}{{\operatorname{Res}}}
\newcommand{\Resultant}{\operatorname{Res}}
\newcommand{\rest}[2]{\left.{#1}\right\vert_{{#2}}}  
\renewcommand{\setminus}{\smallsetminus}
\newcommand{\Span}{\operatorname{Span}}
\newcommand{\Spec}{{\operatorname{Spec}}}
\newcommand{\tors}{{\textup{tors}}}
\newcommand{\Trace}{\operatorname{Trace}}
\newcommand{\UHP}{{\mathfrak{h}}}    

\newcommand{\longhookrightarrow}{\lhook\joinrel\longrightarrow}
\newcommand{\longonto}{\relbar\joinrel\twoheadrightarrow}

\newcommand{\halfquad}{\hspace{.5em}}


\title[\TITLERUNNING]{\TITLE}
\date{\DATE \ifdraft{} --- Draft \DRAFTNUMBER\fi}

\author[Shu Kawaguchi and Joseph H. Silverman]
  {Shu Kawaguchi and Joseph H. Silverman}
\email{kawaguch@math.kyoto-u.ac.jp, jhs@math.brown.edu}
\address{Department of Mathematics, Faculty of Science, Kyoto University, 
Kyoto, 606-8502, Japan}
\address{Mathematics Department, Box 1917
         Brown University, Providence, RI 02912 USA}
\thanks{The first author's research supported by MEXT
grant-in-aid for young scientists (B) 18740008}
\thanks{The second author's research supported by NSA grant H98230-04-1-0064}

\begin{abstract}
Let~$\f:\PP^N_K\to\PP^N_K$ be a morphism of degree~$d\ge2$ defined
over a field~$K$ that is algebraically closed field and complete with
respect to a nonarchimedean absolute value. We prove that a modified
Green function~$\mG_\f$ associated to~$\f$ is H\"older continuous
on~$\PP^N(K)$ and that the Fatou set~$\Fatou(\f)$ of~$\f$ is equal to
the set of points at which~$\mG_\F$ is locally constant.  Further,~$\mG_\f$
vanishes precisely on the set of points~$P$ such that~$\f$ has good
reduction at every point in the forward orbit~$\Ocal_\f(P)$ of~$P$. We
also prove that the iterates of~$\f$ are locally uniformly Lipschitz
on~$\Fatou(\f)$.
\end{abstract}


\maketitle

\section*{Introduction}
\label{section:intro}

Let$K$ be an algebraically closed field that is complete with respect
to a nontrivial nonarchimedean absolute value~\text{$|\,\cdot\,|$}.
An example of such a field is~$\CC_p$, the completion of the algebraic
closure of~$\QQ_p$.  
\par
Let~$\f:\PP_K^1\to\PP_K^1$ be a rational function of degree~$d\ge2$
defined over~$K$.  The absolute value on~$K$ induces a natural metric
on~$\PP^1(K)$, and nonarchimedean dynamics is the study of the
iterated action of~$\f$ on~$\PP^1(K)$ relative to this metric.  The
family of iterates~$\{\f^n\}_{n\ge0}$ divides~$\PP^1(K)$ into two
disjoint (possibly empty) subsets, the Fatou set~$\Fatou(\f)$ and the
Julia set~$\Julia(\f)$.  The Fatou set is the 
the largest open subset of~$\PP^1(K)$ on which the family is 
equicontinuous, and the Julia set is the complement of the Fatou set.  
There has been considerable 
interest in nonarchimedean dynamics on~$\PP^1$ in recent years, see
for example~\cite{%
MR1264116,%
MR1781331,%
MR1941304,%
MR1981035,%
MR1864626,%
MR2031456,%
MR1387667,%
MR1794277,%
MR1769981,%
MR2040006,%
MR2018827,%
MR2156656}.
\par
In this article we investigate aspects of nonarchimedean dynamics on
higher dimensional projective spaces. For points
\[
  P=(x_0:\cdots:x_N)\in \PP^N(K)\quad\text{and}\quad
  Q=(y_0:\cdots:y_N)\in\PP^N(K)
\]
we define the \emph{chordal distance} from $P$ to $Q$ to be
\[
  \D(P, Q) 
  = \frac{\displaystyle\max_{0\leq i, j \leq N}
  |x_i y_j - x_j y_i|}{\max\bigl \{|x_0|,\ldots,|x_N|\bigr\} 
  \max\bigl\{|y_0|,\ldots,|y_N|\bigr \}}.
\]
This defines a nonarchimedean metric on $\PP^N(K)$. As in the one
dimensional case, for any~$K$-morphism $\f: \PP^N \to\PP^N$ of degree
$d\ge2$ we define the Fatou set $\Fatou(\f)$ to be the largest open
set on which the iterates of~$\f$ are equicontinuous, 
and the Julia set $\Julia(\f)$ is the complement of the Fatou set.  (See
Section~\ref{section:fatou:julia} for the precise definitions.)
Also for convenience, for any vector~$x=(x_0,\ldots,x_N)\in K^{N+1}$,
we write~$\|x\|=\max|x_i|$ for the sup norm.
\par
Over the complex numbers, pluri-potential theory has played a key role
in the study of complex dynamics on $\PP^N(\CC)$. One of the primary goals
of this paper is to develop an analogous theory in the nonarchimedean
setting. For a given morphism~$\f:\PP^N_K\to\PP^N_K$ of degree~$d\ge2$,
let
\[
  \F : K^{N+1}\longrightarrow K^{N+1}
\]
be a lift of~$\f$. Then as in the complex case (cf.\ \cite{MR1760844})
one defines the \emph{Green function} (or \emph{potential function}) 
associated to~$\F$ by the limit
\begin{equation}
  \label{eqn:def:G}
  G_\F(x) = \lim_{n\to\infty} \frac{1}{d^n}\log\bigl\|\F(x)\big\|.
\end{equation}
The existence of the limit and the relation of~$G_\F$ to canonical
local height functions is explained in~\cite{KSEqualCanHt}.  We also
define a \emph{modified Green function} 
\begin{equation}
  \label{eqn:modGfun}
  \mG_\F: \PP^N(K) \to \RR, 
  \qquad \mG_\F(P) = G_\F(x) - \log \Vert x \Vert,
\end{equation}
that is well-defined independent of the choice of the lift $x\in
K^{N+1}$ of $P\in\PP^N(K)$.  The main results of this paper are 
summarized in the following theorem.

\begin{theorem}%
\label{thm:mainthmintro}
Let~$\f:\PP^N\to\PP^N$ be a morphism of degree~$d\ge2$ as above and
let~$\mG_\F$ be an associated Green function on~$\PP^N(K)$ as defined
by~\eqref{eqn:def:G} and~\eqref{eqn:modGfun}.
\begin{parts}
\Part{(a)}
The function $\mG_\F$ is H\"older continuous on $\PP^N(K)$. 
\Part{(b)}
The Fatou set of~$\f$ is characterized by
\[
  \Fatou(\f) = \bigl\{ P\in\PP^N(K) : 
    \text{$\mG_\F$ is locally constant at $P$} \bigr\}.
\]
\Part{(c)}
The Fatou set of~$\f$ is equal to the set of points~$P$
such that the iterates of~$\f$ are
locally uniformly Lipschitz at~$P$,
i.e., such that there is a neighborhood~$U$ of~$P$ and
a constant~$C$ so that
\[
  \hspace{1\parindent}
  \dist\bigl(\f^n(Q),\f^n(R)\bigr) \le C \dist(Q,R)
  \quad\text{for all $Q,R\in U$ and all $n\ge0$.}
\]
\Part{(d)}
$\mG_\F(P)=0$ if and only if~$\f$ has good reduction at every point in
the forward orbit~$\Ocal_\f(P)$.  Further, the set of such points is
an open set and is
contained in the Fatou set~$\Fatou(\f)$.
\end{parts}
\end{theorem}

As an immediate corollary of Theorem~\ref{thm:mainthmintro}(b) and the
fact (Corollary~\ref{cor:fisopen}) that~$\f$ is an open mapping in the
nonarchimedean topology, we obtain the invariance of the Fatou and
Julia sets.

\begin{corollary}
The Fatou set~$\Fatou(\f)$ and the Julia set~$\Julia(\f)$ are
forward and backward invariant for~$\f$.
\end{corollary}

\begin{remark}
Parts~(a) and~(b) of Theorem~\ref{thm:mainthmintro} are analogous to
results in pluri-potential theory over~$\CC$. Thus if $\f: \PP_{\CC}^N
\to \PP_{\CC}^N$ is a morphism of degree~$d\ge2$ and $\F: \CC^{N+1}
\to\CC^{N+1}$ is a lift of $\f$, the classical Green function
$G_\F:\CC^{N+1} \to \RR$ associated to $\F$ is defined by the same
limit~\eqref{eqn:def:G} that we are using in the nonarchimedean setting.
It is then well known that~$G_\F$ is H\"older continuous on
$(\CC^{N+1})^*$ and that the Fatou set of~$\f$ is the image
in~$\PP^N(\CC)$ of the set
\[
  \bigl\{x \in (\CC^{N+1})^*:\text{$G_\F$ is pluri-harmonic at $x$} \bigr\}.
\]
See for example \cite{MR1760844}. 
\par
We note that applying~$d d^c$ to~$G_\F$ gives the \emph{Green
current}~$T_\F$ on $\PP^N(\CC)$ and that the invariant measure
associated to~$\f$ is obtained as an intersection of~$T_\F$.  The
invariant measure is of fundamental importance in studying the complex
dynamics of~$\f$. An analogous theory has been developed on~$\PP^1$ in
the nonarchimedean setting (see for
example~\cite{MR2244803,MR2244226,MR2221116,thuillier:thesis})
and it would be interesting to extend this to~$\PP^N$.
\par
Finally, we mention that the H\"older continuity of~$G_\F$ over~$\CC$
is used to estimate the Hausdorff dimension of the Julia set.
\end{remark}

The proof of Theorem~\ref{thm:mainthmintro} is given in
Theorem~\ref{thm:ghatholder},
Theorem~\ref{thm:locally:const:and:F}, and
Proposition~\ref{prop:G0orbitalFatou}.  The proofs of~(a) and~(b)
follow the complex proofs to some extent, but there are also parts of
the proofs that are specifically nonarchimedean, especially where
compactness arguments over~$\CC$ are not applicable to nonlocally
compact fields such as~$\CC_p$. Further, we are able to make most
constants in this article explicit in terms of the Macaulay resultant
of $\Phi$. (See Section~\ref{section:Lipschitz} for the definition and
basic properties of the Macaulay resultant.)

The organization of this paper is as follows. 
In Section~\ref{section:chordal:metric} we define the chordal metric
on $\PP^N(K)$ and prove some of its properties.
In Section~\ref{section:Lipschitz} we consider Lipschitz continuity
and show in particular that $\f: \PP^N \to \PP^N$ is Lipschitz
continuous with an explicit Lipschitz constant.
In Section~\ref{section:properties:green} we review the definition and
basic properties of the Green function~$G_\F$ and use them to deduce
various elementary properties of the modified Green function~$\mG_\F$.
In Section~\ref{section:Holder:continuity} we show that $\mG_\F$ is
H\"older continuous with explicit constants.
In Section~\ref{section:inversefncthm} we prove that morphisms are
open mappings in the nonarchimedean setting.
In Section~\ref{section:preliminaries} we recall some facts from
nonarchimedean analysis.
In Section~\ref{section:fatou:julia} we define the Fatou and Julia
sets in terms of equicontinuity for the family $\{\f^n\}$ with respect
to the chordal metric.
In Section~\ref{section:behavior} we characterize the Fatou
set in terms of the Green function and give some applications,
including the backward and forward invariance of~$\Fatou(\f)$
and~$\Julia(\f)$.
Finally in Section~\ref{section:goodredandfatou} we we relate the
Fatou set and the vanishing of~$\mG_\F$ to sets of points at
which~$\f$ has good reduction.

\begin{acknowledgement}
The authors would like to thank Matt Baker for his assistance. 
The authors would also like to thank Antoine Chambert-Loir,  
Tien-Cuong Dinh and Xander Faber for their helpful comments. 
\end{acknowledgement}


\section{The chordal metric on $\PP^N$}
\label{section:chordal:metric}

For the remainder of this paper we fix an algebraically closed field
field~$K$ that is complete with respect to a nontrivial nonarchimedean
absolute value~\text{$|\,\cdot\,|$}.  We extend the absolute value
on~$K$ to the sup norm on~$K^{N+1}$, which we denote by
\[
  \Vert x \Vert = \max\bigl\{|x_0|, \ldots, |x_N|\bigr\}
  \quad\text{for $x=(x_0, \ldots, x_N) \in K^{N+1}$}.
\]
We also write
\[
  \pi: (K^{N+1})^* \to \PP^N(K)
\]
for the natural projection map.

\begin{definition}
Let $P, Q \in \PP^N(K)$ and choose lifts $x,y\in (K^{N+1})^*$ for~$P$
and~$Q$, i.e., $\pi(x)=P$ and $\pi(y)=Q$.  The (\emph{nonarchimedean})
\emph{chordal distance from~$P$ to~$Q$} is defined by
\[
  \dist(P, Q)
  =
  \frac{\displaystyle\max_{0 \leq i, j \leq N} |x_i y_j - x_j y_i| }%
  {\Vert x \Vert \cdot \Vert y\Vert}.
\]
By homogeneity, it is clear that~$\dist(P,Q)$ is independent of the
choice of lifts for~$P$ and~$Q$.
\end{definition}

\begin{remark}
The chordal distance is an example of a $v$-adic (arithmetic) distance
function as defined in~\cite[\S3]{MR919501}, although we note that the
function~$\d$ defined in~\cite{MR919501} is logarithmic, i.e.,
$\d(P,Q)=-\log\dist(P,Q)$. Further, all of the distance and height
functions in~\cite{MR919501} are Weil functions in the sense that they
are only defined up to addition of a bounded function that depends on
the underlying variety. So to be precise, the logarithmic chordal
distance~$-\log\dist$ is a particular function in the equivalence
class of arithmetic distance functions~$\d$ on~$\PP^N$.
\end{remark}

\begin{lemma}
\label{lemma:spherdist}
The chordal distance $\dist$ defines a nonarchimedean metric on
$\PP^N(K)$.  Further, it is bounded by $\dist(P, Q) \leq 1$.
\end{lemma}

\begin{proof}
It is immediate from the definition that~$\dist(P,Q)\ge0$ and
that it is equal to~$0$ if and only if~$P=Q$. Further, 
\[
 \max_{0\le i,j\le N} |x_i y_j - x_j y_i| 
 \leq \max_{0\le i,j\le N} \max\bigl\{|x_i y_j|, |x_j y_i|\bigr\}
 \leq \Vert x \Vert \cdot \Vert y\Vert,
\]
which proves that~$\dist(P,Q)\le 1$. It remains to verify that~$\dist$
satisfies the strong triangle inequality.
\par
Let $R \in \PP^N(K)$ be a third point and lift it to
$z\in(K^{N+1})^*$. Multiplying each lift by an appropriate element
of~$K^*$, we may normalize the lifts to satisfy
\[
  \Vert x \Vert =  \Vert y \Vert = \Vert z \Vert =1.
\]
Consider the identity
\begin{multline}
  \label{eqn:xizkxkzi}
  (x_i z_k - x_k z_i) y_j  \\
  = (x_i y_j - x_j y_i) z_k +
  (y_i z_k - y_k z_i) x_j +
  (x_j y_k - x_k y_j) z_i.  
\end{multline}
Since $\Vert y \Vert =1$, there is a $j_0$ with $|y_{j_0}| = 1$.
Then~\eqref{eqn:xizkxkzi} with $j = j_0$ gives
\[
  |x_i z_k - x_k z_i| \leq
  \max\{\dist(P, Q), \dist(Q, R)\}.  
\]
Taking the maximum over~$i$ and~$k$ yields the strong triangle
inequality,
\[
  \dist(P, R) \leq \max\{\dist(P, Q), \dist(Q, R)\}.
  \tag*{\qedsymbol}
\]
\renewcommand{\qedsymbol}{}
\end{proof}

In the remainder of this section we develop some basic properties of
the chordal metric on~$\PP^N(K)$.  We begin with some notation that
will be used throughout the remainder of this paper.
\par
Let $M \geq 1$ be an integer, typically equal to either $N$
or~$N+1$. For $a \in K^M$ and $r >0$, the {\em open polydisk} and the
{\em closed polydisk} centered at $a$ with radius $r$ are 
defined, respectively, by
\begin{align*}
  B(a, r) & = \{
  x \in K^M \;:\; \Vert x - a \Vert  <  r 
  \}, \\
  \bar{B}(a, r) & = \{
  x \in K^M \;:\; \Vert x - a \Vert  \leq  r 
  \}.
\end{align*}
Similarly, for $P \in \PP^N(K)$ and $1 \geq r > 0$, we define the 
{\em  open disk} and the {\em closed disk} centered at $P$ with 
radius $r$ to be, respectively,
\begin{align*}
  D_r(P) & = \{
  Q \in \PP^N(K) \ \;:\;  \D(P, Q) < r\}, \\
  \bar{D}_r(P) & = \{
  Q \in \PP^N(K) \ \;:\;  \D(P, Q) \leq r\}.
\end{align*}
Despite the terminology, all four of the
sets~$B(a,r)$,~$\bar{B}(a,r)$,~$D_r(P)$, and~$\bar{D}_r(P)$ are both
open and closed in the topology induced by~$\Vert\,\cdot\,\Vert$ on~$K^M$ and
by the chordal metric~$\dist$ on~$\PP^N(K)$.
We also embed~$K^N$ into~$\PP^N(K)$ via the map
\[
  \sigma: K^N \longhookrightarrow \PP^N(K), 
  \qquad (x_1, \ldots, x_N) \longmapsto (1: x_1: \cdots :x_N). 
\]

\begin{lemma}
\label{lemma:distance:points}
Let $P, Q \in \PP^N(K)$ be points satisfying $\dist(P,Q)<1$.  Choose
a lift~$x\in(K^{N+1})^*$ for~$P$ and a lift~$y\in(K^{N+1})^*$ for~$Q$.  
and let $0\le k\le N$ be an index. Then
\[
  |x_k| = \Vert x \Vert
  \quad\text{if and only if}\quad
 |y_k| = \Vert y \Vert.
\]
\end{lemma} 
\begin{proof}
We may assume that~$\|x\|=\|y\|=1$. Assume that $|x_k|=1$ and
choose an index~$j$ such that~$|y_j|=1$. Then
\[
  |x_ky_j-x_jy_k| \le \dist(P,Q) <  1
  \quad\text{and}\quad |x_ky_j|=1,
\]
so the strong triangle inequality implies that
$|x_jy_k|=1$. But~$|x_j|\le1$ and~$|y_k|\le1$, so we must
have~$|y_k|=1$.
\end{proof}

The next lemma shows that the usual metric~$\Vert \cdot \Vert$ and the
chordal metric~$\Delta$ are the same on the
closed unit polydisk~$\bar{B}(0,1)$ in~$K^N$. 

\begin{lemma}
\label{lemma:comparison:disks}
\textup{(a)}\enspace
The restriction of~$\sigma$ to~$\bar{B}(0, 1)$ is an isometry,
\[
  \D\bigl(\s(x), \s(y)\bigr) = \Vert x - y \Vert  
  \quad\text{for all $x, y\in \bar{B}(0, 1)$}. 
\]
\par\noindent
\textup{(b)}\enspace
Let $x \in \bar{B}(0, 1)$ and  $1> r >0$. Then the maps
\[
  \s:\bar{B}(x,r)\to\bar{D}_r\bigl(\s(x)\bigr)
  \qquad\text{and}\qquad
  \s:{B}(x,r)\to{D}_r\bigl(\s(x)\bigr)
\]
are isometric isomorphisms.
\end{lemma} 
\begin{proof}
Let~$x,y\in\bar{B}(0,1)$.  Then $\Vert \s(x) \Vert = \Vert \s(y) \Vert
=1$, so 
\[
  \D\bigl(\sigma(x), \sigma(y)\bigr) = 
  \max_{0 \leq i,j \leq N}\bigl\{|x_iy_j - x_jy_i|\bigr\},
\]
where for convenience we set~$x_0=y_0=1$. 
In particular, putting~$j=0$ gives
\[
  \D\bigl(\sigma(x), \sigma(y)\bigr) \ge
  \max_{0 \leq i \leq N}\bigl\{|x_i - y_i|\bigr\}
  = \|x-y\|.
\]
Further, we note that
\[
  |x_iy_j - x_jy_i|
  = \bigl| x_i (y_j-x_j) + x_j (x_i - y_i) \bigr|
  \le \max\bigl\{ |y_j-x_j|, |x_i - y_i| \bigr\}.
\]
Taking the maximum over all~$i$ and~$j$ gives
\[
  \D(\sigma(x), \sigma(y)) 
  \le \max_{0\le i\le N} |x_i-y_i|
  = \Vert x - y \Vert,.
\]
which gives the opposite inequality and completes the
proof of~(a).
\par
By assumption $x\in\bar{B}(0,1)$ and $r<1$, so the triangle inequality
implies that~$\bar{B}(x,r)\subset\bar{B}(0,1)$.  Then~(a) tells us that~$\s$
is an isometry on~$\bar{B}(x,r)$, so in particular~$\s$ maps~$\bar{B}(x,r)$
injectively and isometrically into~$\bar{D}_r\bigl(\s(x)\bigr)$.
\par
It remains to check that the map is surjective.  Let~$Q \in
\bar{D}_r(\sigma(x))$ and lift~$Q$ to~$b = (b_0, b_1, \ldots, b_N)$.
We know that~$\|\s(x)\|=1$ and that the first coordinate of~$\s(x)$
equals~$1$, and also $\dist(Q,\sigma(x))\le r<1$, so
Lemma~\ref{lemma:distance:points} tells us the~$|b_0|=\|b\|$.  Then
the point
\[
  y = \left(\frac{b_1}{b_0},\frac{b_2}{b_0},\ldots,\frac{b_N}{b_0}\right)
  \quad\text{is in $\bar{B}(0,1)$ and satisfies $\s(y)=Q$.}
\]
Finally, since~$x,y\in\bar{B}(0,1)$,  we can use~(a) again to compute
\[
  \| x - y \| 
  = \dist\bigl(\s(x),\s(y)\bigr) 
  = \dist\bigl(\s(x),Q\bigr) 
  \le r,
\]
so in fact~$y\in\bar{B}(x,r)$. This proves
that~$\s\bigl(\bar{B}(x,r)\bigr)=\bar{D}_r\bigl(\sigma(x)\bigr)$,
which completes the first part of~(b). The second part is proven
similarly.
\end{proof}

\begin{proposition}
$\PP^N(K)$ is complete with respect to the chordal metric~$\Delta$. 
\textup(As always, we are assuming that the field~$K$ is complete.\textup)
\end{proposition}
\begin{proof}
Fix some~$r<1$, say~$r=\frac12$.  Let~$(P_i)_{i\ge1}$ be a Cauchy
sequence in~$\PP^N(K)$ and fix an~$n$ so that~$\dist(P_i,P_j)\le r$
for all~$i,j\ge n$. In particular, the truncated sequence~$(P_i)_{i\ge
n}$ lies in the disk~$\bar{D}_r(P_n)$. Reordering the coordinates if
necessary, we can assume that there is a lift~$x\in\bar{B}(0,1)$
of~$P_n$. Then Lemma~\ref{lemma:comparison:disks}(b) tells us
that~$\bar{D}_r(P_n)$ is isometrically isomorphic
to~$B(x,r)$. But~$B(x,r)\subset K^N$ and~$K^N$ is complete,
hence~$\bar{D}_r(P_n)$ is also complete.
\end{proof}

\section{Lipschitz continuity of morphisms}
\label{section:Lipschitz}

In this and subsequent sections, we say that an element~$a\in K$ is
\emph{$K$-integral} if~$|a|\le1$ and we say that~$a$ is a
\emph{$K$-unit} if~$|a|=1$.

Associated to any collection of homogeneous polynomials
\[
  \F=(\F_0,\ldots,\F_N):\AA^{N+1}\longrightarrow\AA^{N+1}
\]
in~$N+1$~variables is a polynomial~$\Resultant(\F)$ (with integer
coefficients) in the coefficients of~$\F_0,\ldots,\F_N$ whose
vanishing is equivalent to the collection $\F_0,\ldots,\F_N$ having a
nontrivial common zero.  See~\cite[\S1.1]{KSEqualCanHt} for a summary
of the basic properties of this \emph{Macaulay
resultant}~$\Resultant(\F)$ and~\cite{MR1142904} for full details and
proofs.  We recall the following useful result.

\begin{proposition}
\label{prop:ResFxdleFlexd}
Let~$\F_0,\ldots,\F_N\in K[X_0,\ldots,X_N]$ be a 
collection of homogeneous polynomials
with $K$-integral coefficients.  Then
\[
  |\Resultant(\F)|\cdot\|x\|^d \le \|\F(x)\| \le \|x\|^d
  \qquad\text{for all $x\in \AA^{N+1}(K)$.}
\]
\end{proposition}
\begin{proof}
See~\cite[Proposition~6(b)]{KSEqualCanHt}.
\end{proof}

\begin{definition}
Let $\f: \PP^N_K \to \PP^N_K$ be a morphism defined over~$K$ and let
$\F: \AA^{N+1}_K \to \AA^{N+1}_K$ be a lift of~$\f$. We say that~$\F$
is a \emph{minimal lift of~$\f$} if all of its coefficients are
$K$-integral and at least one coefficient is a $K$-unit. Any two minimal lifts
differ by multiplication by a $K$-unit.
\par
We define a \emph{minimal resultant~$\Resultant(\f)$ of~$\f$} to be
the resultant of a minimal lift of~$\f$. Note that~$\Resultant(\f)$ is
well defined up to multiplication by a power of a $K$-unit, so in
particular, the absolute value $|\Resultant(\f)|$ is well defined
independent of the chosen minimal lift.
\end{definition}

\begin{definition}
Let $\F=(\F_0,\ldots,\F_N):K^{N+1}\to K^{N+1}$ be 
a lift of~$\f:\PP^N_K\to\PP^N_K$. For each $i = 0, \ldots, N$, 
we define the \emph{norm of~$\F_i$} to be the 
maximum of the absolute values of the coefficients of~$\F_i$. 
In other words, if $\F_i = \sum
a_{i,j_0,\dots,j_N}x_0^{j_0}\cdots x_N^{j_N}$, then 
\[
  \|\F_i\| = \sup_{j_0,\ldots,j_N\ge0} 
             \left|a_{i,j_0,\dots,j_N}\right|. 
\]
We define the \emph{norm of~$\F$} by $\|\F\| = \sup_{0\le i\le N} \|\F_i\|$. 
In particular, the condition~$\|\F\|=1$ is equivalent to~$\F$ being a
minimal lift of~$\f$. 
\end{definition}

We now prove that morphisms of~$\PP^N$ over nonarchimedean fields are
Lipschitz continuous and give an explicit Lipschitz constant.

\begin{theorem}
\label{theorem:f:Lipschitz}
Let
$\f: \PP^N \to \PP^N$ be a morphism of degree $d \geq 2$ defined over~$K$.  
Then~$\f$ is Lipschitz continuous with respect to the chordal metric.  More
precisely,
\begin{equation}
\label{eqn:f:Lipschitz}
  \dist(\f(P), \f(Q)) \leq |\Resultant(\f)|^{-2} \dist(P, Q)
  \qquad\text{for all $P, Q \in \PP^N(K)$},
\end{equation}
where~$\Resultant(\f)$ is a minimal resultant of~$\f$.
\end{theorem}

\begin{remark}
More generally, any morphism $\f:\PP^N\to\PP^M$ is Lipschitz
continuous, although the Lipschitz constant depends in a more
complicated way on~$\f$.  
\end{remark}

\begin{remark}
\label{remark:good:reduction}
Recall that the map~$\f$ has good reduction if its minimal resultant is
a $K$-unit.  (See~\cite[Section~1.3]{KSEqualCanHt}.) Hence if~$\f$ has
good reduction, then~$\f$ is nonexpanding with respect to the
chordal metric, so the Julia set of~$\f$ 
(see Section~\ref{section:fatou:julia}) is empty. This generalizes
the well-known result for~$\PP^1$, see for example~\cite{MR1324210}.
\end{remark}

\begin{proof}[Proof of Theorem \ref{theorem:f:Lipschitz}]
Let~$\F=(\F_0:\cdots:\F_N)$ be a minimal lift of~$\f$.
Consider the homogeneous polynomials
\[
  \F_i(X)\F_j(Y)-\F_j(X)\F_i(Y)\in K[X,Y].
\]
They are in the ideal generated by 
\[
  \bigl\{ X_kY_l-X_lY_k : 0 \le k < l\le N \bigr\}.
\]
More precisely, there are polynomials~$A_{i,j,k,l}(X,Y)$ whose
coefficients are bilinear forms (with integer coefficients) in the
coefficients of~$\F_i$ and~$\F_j$ such that
\[
  \F_i(X)\F_j(Y)-\F_j(X)\F_i(Y)
  = \sum_{0\le k <l\le N} A_{i,j,k,l}(X,Y)(X_kY_l-X_lY_k).
\]
\par
Now let $P,Q\in\PP^N(K)$ and write~$P=\pi(x)$ and~$Q=\pi(y)$ as usual
with $\|x\|=1$ and $\|y\|=1$. Then
\begin{align}
  \label{eqn:FiFjFjFi}
  \bigl|\F_i(x)\F_j(y)&-\F_j(x)\F_i(y)\bigr| \notag\\
  &\le
  \max_{0\le k <l\le N} \bigl|A_{i,j,k,l}(x,y)|\cdot\bigl|x_ky_l-x_ly_k\bigr| 
    \notag\\
  &\le  \|\F_i\|\cdot\|\F_j\|
      \max_{0\le k <l\le N} |x_ky_l-x_ly_k| \notag\\
  &\le \dist(P,Q)
   \qquad\text{since $\|\F\|=1$ by assumption.}
\end{align}
\par
Since~$\F$ has $K$-integral coefficients and~$\|x\|=\|y\|=1$, 
Proposition~\ref{prop:ResFxdleFlexd} says that
\begin{equation}
  \label{eqn:FxgeRF}
  \|\F(x)\| \ge |\Resultant(\F)|
  \qquad\text{and}\qquad
  \|\F(y)\| \ge |\Resultant(\F)|.
\end{equation}
Using~\eqref{eqn:FiFjFjFi} and~\eqref{eqn:FxgeRF} in the definition
of the chordal distance yields
\begin{align*}
  \dist\bigl(\f(P),\f(Q)\bigr)
  &= \frac{\displaystyle\max_{0 \leq i, j \leq N} 
       \bigl|\F_i(x)\F_j(y)-\F_j(x)\F_i(y)\bigr|}
    {\Vert \F(x) \Vert \cdot \Vert \F(y)\Vert} \\
  &\le |\Resultant(\F)|^{-2}\dist(P,Q).
\end{align*}
This completes the proof of Theorem~\ref{theorem:f:Lipschitz}.
\end{proof}

The previous theorem considered the distance from~$\F(P)$ to~$\F(Q)$. 
We next study the variation of the ratio of~$\|\F(P)\|$ to~$\|\F(Q)\|$.
  
\begin{theorem}
\label{theorem:g:Lipschitz}
Let $\f: \PP^N \to \PP^N$ be a morphism of degree $d \geq 2$ defined
over~$K$, let~$\F:\AA^{N+1}_K\to\AA^{N+1}_K$ be a lift of~$\f$, and
define a function
\begin{align}
  \label{eqn:defgdFx}
  g_\F:\PP^N(K)\longrightarrow\RR,\qquad
  g_\F(P) = \smash[b]{\frac{1}{d}} \log\Vert \F(x)\Vert &- \log\Vert x\Vert \\
  &\text{for any~$x\in\pi^{-1}(P)$.} \notag
\end{align}
Then $g_\F$ is Lipschitz continuous with respect to the chordal
metric. 
\par
More precisely,  for all $P, Q \in \PP^N(K)$ we have
\begin{equation}
\label{eqn:g:Lipschitz}
  \bigl|g_\F(P) -  g_\F(Q)\bigr| 
     \leq \frac{\log(|\Resultant(\f)|^{-1})}{d|\Resultant(\f)|}\dist(P,Q).
\end{equation}
Further,
\begin{equation}
\label{eqn:g:locconst}
  g_\F(P)=g_\F(Q)\qquad
  \text{if\quad $\dist(P,Q)<|\Resultant(\f)|$.}
\end{equation}
In particular,~$g_\F$ is \emph{uniformly} locally constant.  \textup{(}Note that
the norm on the lefthand side of~\eqref{eqn:g:Lipschitz} is the usual
archimedean absolute value on $\RR$.\textup{)}
\end{theorem}

\begin{proof}
Homogeneity of~$\F$ implies that~$g_\F(P)$ is well-defined, independent
of the lift of~$P$. Further, for any constant~$c$ we have 
\[
  g_{c\F}(P)=g_\F(P) + \frac{1}{d}\log|c|,
\]
so the difference~$g_\F(P)-g_\F(Q)$ is independent of the chosen lift
of~$\f$.  Hence without loss of generality, we assume that~$\F$ is a
minimal lift of~$\f$. To ease notation, we let
\[
  R = \bigl|\Resultant(\f)\bigr|
\]
be the absolute value of the minimal resultant. Note
that~\text{$0<R\le1$}.
\par
Let~$P=\pi(x)$ and~$Q=\pi(y)$ with~$\|x\|=\|y\|=1$ as usual,
so in particular Proposition~\ref{prop:ResFxdleFlexd}
tells us that
\begin{equation}
  \label{eqn:FxgeRF1}
  1 \ge \|\F(x)\| \ge R
  \qquad\text{and}\qquad
  1\ge \|\F(y)\| \ge R.
\end{equation}
\par
We consider two cases. The first case is for points~$P$ and~$Q$ that are not
close together. Suppose that~$\dist(P,Q)\ge R$. Then
using~\eqref{eqn:FxgeRF1} we find that
\[
  \bigl|g_\F(P)-g_\F(Q)\bigr|
  =\frac{1}{d}\left|\log\frac{\|\F(x)\|}{\|\F(y)\|} \right|
  \le \frac{1}{d}\log(R^{-1})
  \le \frac{\log(R^{-1})}{dR}\dist(P,Q).
\]
This proves that the function~$g_\F$ is Lipschitz for points~$P$ and~$Q$
satisfying~$\dist(P,Q)\ge R$.
\par
Next we consider the case that $\dist(P,Q)<R$. Notice the strict
inequality, so in particular~$\dist(P,Q)<1$. We
have~\text{$\|x\|=\|y\|=1$} by assumption, so from
Lemma~\ref{lemma:distance:points} we can find an index~$k$ such
that \text{$|x_k|=|y_k| =1$}.  
\par 
In order to complete the proof, we expand $\F(x+h)$ as
\[
  \F(x+h) = \F(x) + \sum_{i=0}^N h_i B_i(x,h),
\]
where each~$B_i$ is a vector of polynomials whose coefficients are
linear forms (with integer coefficients) in the coefficients
of~$\F$. Then using the particular index~$k$ determined above, we
compute
\begin{align}
  \label{eqn:Fx=Fxky}
  \|\F(x)\|
  &=\|y_k^d\F(x)\| \notag\\
  &=\|\F(y_kx)\| \notag\\
  &=\|\F(x_ky+y_kx-x_ky)\| \notag\\
  &=\Bigl\|\F(x_ky) + \sum_{i=0}^N (y_kx_i-x_ky_i) B_i(x_ky,y_kx-x_ky)\Bigr\|.
\end{align}
Now we observe that
\begin{align*}
  \Bigl\|\sum_{i=0}^N (y_kx_i-x_ky_i) B_i(x_ky,y_kx-x_ky)\Bigr\|
  & \le \max_i |y_kx_i-x_ky_i| \\
  & \le \dist(P,Q)
  < R,
\end{align*}
while in the other direction we have
\[
  \|\F(x_ky)\|=|x_k|^d\|\F(y)\|=\|\F(y)\|\ge R.
\]
Hence the first term in the righthand side of~\eqref{eqn:Fx=Fxky} has
absolute value strictly larger than the second term, so we deduce that
\[
  \|\F(x)\|=\|\F(x_ky)\|=|x_k|^d\|\F(y)\|=\|\F(y)\|.
\]
Hence
\[
  g_\F(P)-g_\F(Q) = \frac{1}{d}\log\frac{\|\F(x)\|}{\|\F(y)\|}=0.
\]
We have thus proven that if $\dist(P,Q)<R$, then~$g_\F(P)=g_\F(Q)$,
which completes the proof of Theorem~\ref{theorem:g:Lipschitz}.
\end{proof}

\section{Elementary properties of the Green function}
\label{section:properties:green}

In this section we recall from~\cite{KSEqualCanHt} the definition and
basic properties of nonarchimedean Green functions. Note that what we call 
{\em nonarchimedean Green functions} are called {\em 
homogeneous local canonical height functions} in \cite{MR2244226}, 
and the (Arakelov) Green functions in \cite{MR2244226} are functions 
on $\PP^1\times\PP^1$ with a logarithmic pole along the diagonal. 

\begin{theorem}
\label{thm:greenfuncprops}
Let $\f: \PP_K^N \to \PP_K^N$ be a morphism of degree $d \geq 2$ 
and let $\F: K^{N+1}\to K^{N+1}$ be a lift of~$\f$.  
\begin{parts}
\Part{(a)}
There is a unique function
\[
  G_\F: (K^{N+1})^* \longrightarrow \RR
\]
satisfying 
\begin{equation}
\label{eqn:properties:G}
  G_\F\bigl(\F(x)\bigr) = d G_\F(x) 
  \quad\text{and}\quad 
  G_\F(x) = \log \Vert x \Vert + O(1).
\end{equation}
The function~$G_\F$ is called the \emph{Green function of~$\F$}.
\Part{(b)}
The value of the Green function is given by the limit
\[
  G_\F(x) = \lim_{n\to\infty} \frac{1}{d^n} \log\Vert \F^n(x) \Vert.
\]
\Part{(c)}
The Green function satisfies
\[
  \hspace{1\parindent}
  G_\F(cx) = G_\F(x)+\log|c|
  \qquad\text{for all $c\in K^*$ and all~$x\in(K^{N+1})^*$.}
\]
\Part{(d)}
If we use a different lift~$c\F$ in place of~$\F$,
then the Green function changes by a constant amount,
\[
  G_{c\F}(x) = G_\F(x) +\frac{1}{d-1}\log|c|.
\]
\end{parts}
\end{theorem}
\begin{proof}
See~\cite[Theorem 7]{KSEqualCanHt} for~(a,b,c)
and~\cite[Lemma~8]{KSEqualCanHt} for~(d).
\end{proof}

\begin{definition}
Let~$\f:\PP^N_K\to\PP^N_K$ be a morphism of degree~$d\ge2$,
let~$\F$ be a lift of~$\f$, and let~$G_\F$ be the associated Green
function. We define the (\emph{modified}) \emph{Green function of~$\f$}
to be the function
\begin{equation}
  \label{eqn:ghat}
  \begin{aligned}
     &\mG_\F:\PP^N(K)\longrightarrow\RR,\\
     &\mG_\F(P) = G_\F(x) - \log \Vert x\Vert
     \quad\text{for any $x\in\pi^{-1}(P)$.}
  \end{aligned}
\end{equation}
\end{definition}

We end this section by proving a few elementary properties
of the modified Green function.

\begin{proposition}
\label{prop:modgreenfuncprops}
Let $\f: \PP_K^N \to \PP_K^N$ be a morphism of degree $d \geq 2$,
let $\F: K^{N+1}\to K^{N+1}$ be a lift of~$\f$, and let~$\mG_\F$
be the modified Green function defined by~\eqref{eqn:ghat}.
\begin{parts}
\Part{(a)}
$\mG_\F(P)$ does not depend on the choice of the lift~$x\in
K^{N+1}$ of~$P$, so~$\mG_\F$ is a well-defined function on~$\PP^N(K)$.
\vspace{1\jot}
\Part{(b)}
Let $g_\F(P) = d^{-1}\log\bigl\|\F(x)\bigr\|-\log\|x\|$ be the
function defined by~\eqref{eqn:defgdFx} in the statement of
Theorem~\textup{\ref{theorem:g:Lipschitz}}. Then
\[
  \mG_\F\bigl(\f(P)\bigr)
  = d\mG_\F(P) - dg_\F(P).
\]
\Part{(c)}
The Green function~$\mG_\F$  is given by the series
\[
  \mG_\F(P) = \sum_{n=0}^\infty \frac{1}{d^n} g_\F\bigl(\f^n(P)\bigr).
\]
\Part{(d)}
Assume that~$\F$ is a minimal lift of~$\f$. Then
the Green function~$\mG_\F$ is nonpositive.
Further,~$\mG_\F(P)=0$ if and only if~$g_\F\bigl(\f^n(P)\bigr)=0$ for
all~$n\ge0$.  \textup(See Theorem~\textup{\ref{prop:G0orbitalFatou}}
for a characterization of the set where~$\mG_\F(P)=0$.\textup)
\end{parts}
\end{proposition}
\begin{proof}
(a) The homogeneity of the Green function
(Theorem~\ref{thm:greenfuncprops}(c))  implies that
\[
  G_\F(cx)-\log\|cx\| = G_\F(x)-\log(x)
  \qquad\text{for all $c\in K^*$.}
\]
\par\noindent(b)\enspace
The transformation property for~$G_\F$ 
(Theorem~\ref{thm:greenfuncprops}(a)) gives 
\begin{align*}
  \mG_\F\bigl(\f(P)\bigr) 
  &= G_\F\bigl(\F(x)\bigr) - \log\bigl\|\F(x)\bigr\| \\
  &= d G_\F(x) - \log\bigl\|\F(x)\bigr\| \\
  &= d \mG_\F(P) - \bigl(\log\bigl\|\F(x)\bigr\| - d\log\|x\| \bigr) \\
  &= d \mG_\F(P) - d g_\F(P).
\end{align*}
\par\noindent(c)\enspace
This follows from  the usual telescoping sum argument. Thus
\begin{align*}
  \sum_{n=0}^k \frac{1}{d^n}g_\F\bigl(\f^n(P)\bigr)
  &= \sum_{n=0}^k \frac{1}{d^n}
        \Bigl( \frac{1}{d}\log\bigl\|\F^{n+1}(x)\bigr\|
           - \log\bigl\|\F^{n}(x)\bigr\| \Bigr) \\
  &= \frac{1}{d^k}\log\bigl\|\F^{k}(x)\bigr\| - \log\|x\|.
\end{align*}
Letting~$k\to\infty$, the righthand side goes to~$\mG_\F(P)$.
\par\noindent(d)\enspace
The upper bound in Proposition~\ref{prop:ResFxdleFlexd} tells us that
the function~$g_\F$ satisfies
\[
  g_\F(P) = \frac{1}{d}\log\bigl\|\F(x)\bigr\|-\log\|x\|\le 0.
\]
Hence the sum in~(c) consists entirely of nonpositive terms.
It follows that~$\mG_\F(P)\le 0$, and further~$\mG_\F(P)=0$ if and only
if every term in the sum vanishes.
\end{proof}

\begin{remark} 
Chambert-Loir tells us that the modified Green function~$\mG_\F$ is 
related to the canonical [admissible] metric on 
the line bundle~$\Ocal_{\PP^N}(1)$ 
introduced by Zhang \cite{MR1311351}. 

Precisely, we write~$\Vert\cdot\Vert_{\sup}$ for 
the metric on~$\Ocal_{\PP^N}(1)$ defined by  
\[
  \Vert s\Vert_{\sup}(P) = \frac{|s(x)|}{\Vert x \Vert} 
  \quad\text{for~$s \in \Gamma(\PP^N, \Ocal_{\PP^N}(1))$ 
  and any~$x \in \pi^{-1}(P)$}, 
\]
and~$\Vert\cdot\Vert_{\F}$ for the canonical metric 
on~$\Ocal_{\PP^N}(1)$ associated to~$\f:\PP^N\to\PP^N$ and 
a lift $\F$ of $\f$ (see \cite[Theorem~(2.2)]{MR1311351}). Then we obtain
\[
  \mG_\F = \log\frac{\Vert\cdot\Vert_{\sup}}{\Vert\cdot\Vert_{\F}}. 
\]
Hence properties of~$\mG_\F$ give the 
corresponding properties of the canonical metric 
$\Vert\cdot\Vert_{\F}$. 
\end{remark}

\section{H{\"o}lder continuity of the Green function}
\label{section:Holder:continuity}

Our goal in this section is to prove that~$\mG_\F$ is H{\"o}lder
continuous on~$\PP^N$.  We follow the argument of
Dinh--Sibony~\cite[Proposition~2.4]{MR2137979} (See also 
Favre--Rivera-Letelier~\cite[Proposition~6.5]{MR2221116}).   
Over a nonarchimedean valuation field, we easily obtain explicit 
constants for H{\"o}lder continuity. 
We begin with an elementary lemma.  

\begin{lemma}
\label{lemma:minDakbk}
Let~$a,b,D$ be constants satisfying $a >1$, $b >1$ and $0<D\le1$. Then
\[
  \min \{ Da^k + b^{-k} : k \in \ZZ,\, k>0 \}
 \leq 2 a D^{\frac{\log b}{\log a b}}.
\]
\end{lemma}
\begin{proof}
Let~$t\in\RR$ be the number
\[
  t = \frac{\log (D^{-1})}{\log ab}.
\]
Then the assumptions on~$a,b,D$ imply that~$t\ge0$, and by definition
of~$t$ we have $Da^{t} = b^{-t}$. Hence
\[
  D a^{t} + b^{-t}
   = 2 D a^{\frac{\log(D^{-1})}{\log ab}}  
   = 2 D\cdot D^{-\frac{\log a}{\log ab}}  
   = 2 D^{\frac{\log b}{\log ab}}.
\]
We put $k = \lfloor{t}\rfloor + 1$. Then $k$ is a positive integer and
we have
\begin{align*}
  D a^{k} + b^{-k}
   = a^{k -t} D a^{t} + b^{-(k -t)} b^{-t} 
  & \leq a D a^{t} + b^{-t} \\
  & \leq a \left( D a^{t} + b^{-t}\right)
   = 2a D^{\frac{\log b}{\log ab}}.
\end{align*}
This completes the proof of the lemma.
\end{proof}

We now prove that the nonarchimedean Green function is H{\"o}lder
continuous and give explicit H{\"o}lder constants.

\begin{theorem}
\label{thm:ghatholder}
The modified Green function $\mG_\F: \PP^N(K)\to\RR$ defined
by~\eqref{eqn:ghat} is H{\"o}lder continuous. More precisely, let
\[
  u = u(\f) = \max\bigl\{2d,|\Resultant(\f)|^{-2}\bigr\}.
\]
Then
\begin{equation}
  \label{eqn:Holder}
  \bigl| \mG_\F(P) - \mG_\F(Q) \bigr| 
  \le \frac{2u\log u}{d}
    \dist(P,Q)^{\frac{\log d}{\log u}}
  \quad\text{for all~$P,Q\in\PP^N(K)$.}
\end{equation}
\end{theorem}

\begin{proof}
In general, the Green function~$G_\F$ and the modified Green
function~$\mG_\F$ depend on the chosen lift~$\F$ of~$\f$. However,
Theorem~\ref{thm:greenfuncprops}(d) tells us that~$G_{c\F}-G_\F$ is
constant, so the difference~$\mG_\F(P)-\mG_\F(Q)$ is independent of the
chosen lift~$\F$ of~$\f$. Hence without loss of generality we may
assume that~$\F$ is a minimal lift of~$\f$.
\par
To ease notation, we let $R=|\Resultant(\f)|$ as usual. We also recall
the function
\[
  g_\F(P) = \frac{1}{d}\log\|\F(x)\|-\log\|x\|
\]
used in Theorem~\ref{theorem:g:Lipschitz}. Note that
Proposition~\ref{prop:ResFxdleFlexd} tells us that~$g_\F$ is a bounded
function,
\begin{equation}
  \label{eqn:1dRlegPle0}
  \frac{\log(R)}{d} \le g(P) \le 0
  \qquad\text{for all $P\in\PP^N(K)$.}
\end{equation}
Further, Proposition~\ref{prop:modgreenfuncprops}(c)
says that we can write~$\mG_\F$ as a telescoping sum,
\[
  \mG_\F(P) = \sum_{n=0}^{\infty} \frac{1}{d^n} g_\F\bigl(\f^n(P)\bigr).
\]
Let $k$ be an auxiliary integer to be chosen later.
We compute
{\allowdisplaybreaks
\begin{align}
  \label{eqn:whgbd}
  \bigl| \mG_\F(&P) - \mG_\F(Q) \bigr| \notag\\
  &= \Bigl|\sum_{n=0}^\infty \frac{1}{d^n}\bigr(g_\F(\f^n(P))
               -g_\F(\f^n(Q))\bigr)\Bigr| \notag\\
  & \leq \sum_{n=0}^{k-1} \frac{1}{d^n}   \bigl|g_\F(\f^n(P))
        - g_\F(\f^n(Q))\bigr| 
      + 2\Bigl(\sum_{n=k}^{\infty} \frac{1}{d^n}\Bigr)
              \sup_{T\in\PP^N(K)} \bigl|g_\F(T)\bigr|\notag\\
  & \leq \sum_{n=0}^{k-1} \frac{1}{d^n} \cdot \frac{\log(R^{-1})}{dR}
           \dist\bigl(\f^n(P),\f^n(Q)\bigr)
      + \frac{2}{d^k}\cdot\frac{1}{1-d^{-1}}\cdot \frac{\log(R^{-1})}{d} \notag\\
  &\omit\hfil from Theorem~\ref{theorem:g:Lipschitz} 
                 and \eqref{eqn:1dRlegPle0}, \notag\\
  & \leq \sum_{n=0}^{k-1} \frac{1}{d^n} \cdot \frac{\log(R^{-1})}{dR} 
        \cdot R^{-2n}\dist(P,Q)
      + \frac{2\log(R^{-1})}{d-1}\cdot\frac{1}{d^k} \notag\\
  &\omit\hfil from Theorem~\ref{theorem:f:Lipschitz}, \notag\\
  & \leq 2\log(R^{-1})\cdot 
     \left( \frac{\dist(P,Q)}{2}\cdot
     \sum_{n=1}^{k} \frac{1}{(dR^2)^n} + \frac{1}{d^k}
       \right). 
\end{align}
}
\par
The most interesting case is when $dR^2$ is small, say $dR^2\le\frac12$,
so we consider that case first. Then the bound~\eqref{eqn:whgbd} yields
\[
  \bigl| \mG_\F(P) - \mG_\F(Q) \bigr| 
  \le 2\log(R^{-1})\cdot
     \left( \dist(P,Q)\left(\frac{1}{dR^2}\right)^{k} + d^{-k} \right).
\]
We now choose~$k$ as described in Lemma~\ref{lemma:minDakbk}. This gives
the desired upper bound
\begin{equation}
  \label{eqn:Holder1}
  \bigl| \mG_\F(P) - \mG_\F(Q) \bigr| 
  \le \frac{4\log(R^{-1})}{dR^2} \cdot
    \dist(P,Q)^{\frac{\log d}{\log(R^{-2})}}
\end{equation}
\par
Next we suppose that~$dR^2\ge\frac12$. Then $\sum_{n=1}^k (dR^2)^{-n}<2^{k+1}$,
so~\eqref{eqn:whgbd} gives
\[
  \bigl| \mG_\F(P) - \mG_\F(Q) \bigr| 
  \le \log2d \cdot
     \left( \dist(P,Q)2^{k} + d^{-k} \right).
\]
Now another application of Lemma~\ref{lemma:minDakbk} yields the upper bound
\begin{equation}
  \label{eqn:Holder2}
  \bigl| \mG_\F(P) - \mG_\F(Q) \bigr| 
  \le 4\log2d \cdot
    \dist(P,Q)^{\frac{\log d}{\log2d}}.
\end{equation}
Combining~\eqref{eqn:Holder1} and~\eqref{eqn:Holder2} completes the
proof that $\mG_\F$ is H{\"o}lder continuous with the explicit
constants listed in~\eqref{eqn:Holder}.
\end{proof}


\section{Distance functions and the open mapping property}
\label{section:inversefncthm}
In this section we recall a distribution relation for distance
functions proven in~\cite{MR919501}, where it was used to
prove a quantitative nonarchimedean inverse function theorem. We apply
the distribution relation to give a short proof that finite morphisms
$\f:\PP^N\to\PP^N$ over nonarchimedean fields are open maps, i.e.,
they map open sets to open sets. More generally, the same is true for
any finite morphism of projective varieties.

\begin{proposition}[Distribution Relation]
\label{prop:distributionrelation}
Let~$\f:\PP^N\to\PP^N$ be a  morphism of degree~$d\ge1$ defined
over~$K$ and let~$P,T\in\PP^N(K)$. Then
\[
  -\log\dist\bigl(\f(P),T\bigr)
  = \sum_{Q\in\f^{-1}(T)} -e_\f(Q) \log\dist(P,Q) + O_\f(1),
\]
where~$e_\f(Q)$ is the ramification index of~$\f$ at~$Q$ and the
big-$O$ constant depends on~$\f$, but is independent of~$P$ and~$Q$.
\par
In particular, there is a constant $c=c(\f)\ge1$ such that
for all~$P,T\in\PP^N(K)$ we have
\begin{equation}
  \label{eqn:minQfR}
  \min_{Q\in\f^{-1}(T)} \dist(P,Q) \le c\dist\bigl(\f(P),T\bigr)^{1/d}.
\end{equation}
\end{proposition}
\begin{proof}
The first statement is a special case of
\cite[Proposition~6.2(b)]{MR919501}. Note that since~$\PP^N$ is
projective and~$\f$ is defined on all of~$\PP^N$, we do not need
the~$\lambda_{\partial W\times V}$ term that appears
in~\cite{MR919501}. The second statement is immediate from
exponentiating the first statement and using the fact
that~$\sum_{Q\in\f^{-1}(T)}e_\f(Q)=d$.
\end{proof}

\begin{remark}
For refined calculations, there is a version
of~\eqref{eqn:minQfR} without the~$1/d$ exponent provided that~$P$ is
not in the ramification locus of~$\f$. More precisely,
\cite[Theorem~6.1]{MR919501} implies that if~$\f$ is
unramified at~\text{$P\in\PP^N(K)$}, then there is a disk~$D_r(P)$
around~$P$ such that the map
\[
  \f : D_r(P) \longrightarrow \f\bigl(D_r(P)\bigr)
\]
is bijective and biLipschitz, i.e., both~$\f$ and~$\f^{-1}$ are
Lipschitz. Of course, we have already seen that~$\f$ is Lipschitz
(Theorem~\ref{theorem:f:Lipschitz}), the new information is
that~$\f^{-1}$ is also Lipschitz.  Notice that even if~$\f$ is
ramified at~$P$, Proposition~\ref{prop:distributionrelation}
more-or-less says that~$\f^{-1}$ (which doesn't quite exist) satisfies
$\dist\bigl(\f^{-1}(P),\f^{-1}(Q)\bigr)\ll\dist(P,Q)^{1/d}$,
so~$\f^{-1}$ is locally H\"older continuous.
\end{remark}

\begin{corollary}
\label{cor:fisopen}
Let~$\f:\PP^N\to\PP^N$ be a  morphism of degree~$d\ge1$ defined over~$K$.
Then~$\f$ is an open mapping, i.e.,~$\f$ maps open sets to open sets.
\end{corollary}
\begin{proof}
Let~$U\subset\PP^N(K)$ be an open set and let~$\f(P)\in \f(U)$ be a
point in the image of~$\f$. We need to find a disk around~$\f(P)$
that is contained in~$\f(U)$. Since~$U$ is open, we can find an~$\e>0$
so that~$D_\e(P)\subset U$. Let~$\d=(\e/c)^d$, where~$c$ is the constant
appearing in~\eqref{eqn:minQfR} in
Proposition~\ref{prop:distributionrelation}. We claim
that~$D_\d\bigl(\f(P)\bigr)\subset\f(U)$, which will complete the proof.  
\par
So let~$T\in D_\d\bigl(\f(P)\bigr)$. We apply the second statement in
Proposition~\ref{prop:distributionrelation} to find a point~$Q\in\f^{-1}(T)$
satisfying
\[
  \dist(P,Q) 
  \le c\dist\bigl(\f(P),T)^{1/d}
  < c \d^{1/d}
  = \e.
\]
Hence~$Q\in D_\e(P)\subset U$, so $T=\f(Q)\in\f(U)$.
\end{proof}

We note that the H\"older-type inequality~\eqref{eqn:minQfR} that
follows from the distribution relation
(Proposition~\ref{prop:distributionrelation}) can be used to prove
directly from the definition that the Fatou and Julia sets of~$\f$
are completely invariant. However, since we have not yet defined the
Fatou and Julia sets, we defer the proof until
Section~\ref{section:behavior}, where we instead give a short proof
based on our characterization of the Fatou set as the set 
on which the Green function is locally constant.



\section{Nonarchimedean analysis}
\label{section:preliminaries}
Let $K$ be an algebraically closed field that is complete with respect
to a nonarchimedean absolute value as usual.  In this section, we
recall some basic facts from nonarchimedean analysis. For details we
refer the reader to \cite{MR0746961}.

Let $a = (a_1, \ldots, a_N) \in K^N$ and let $r\in|K^*|$ be a real
number in the value group of~$K$.  A formal power series
\[
  \Psi(x) = \sum_{i_1, \ldots, i_N \geq 0} 
  c_{i_1\ldots i_N} (x_1-a_1)^{i_1}\cdots (x_N-a_N)^{i_N}
\] 
is said to be {\em analytic on $\bar{B}(a, r)$} if the coefficients
$c_{i_1\ldots i_N} \in K$ satisfy
\[
  \lim_{i_1 + \cdots + i_N \to \infty}
  |c_{i_1\ldots i_N}| r^{i_1 + \cdots + i_N} =  0.
\]
Then~$\Psi(x)$ defines a function~$\Psi:\bar{B}(a,r)\to K$.
The \emph{Gauss norm of~$\Psi$} on~$\bar{B}(a,r)$ is the quantity
\[
  \Vert \Psi \Vert_{\bar{B}(a, r)} 
  = \sup_{i_1\ldots i_N} 
  \{ |c_{i_1\ldots i_N}| r^{i_1+ \cdots +  i_N}\}.
\]
If~$\Psi$ is analytic on~$\bar{B}(a,r)$, then~$\Vert \Psi \Vert_{\bar{B}(a, r)}$
is finite, and the strong triangle inequality gives
\[
  |\Psi(x)| \leq \Vert \Psi \Vert_{\bar{B}(a, r)}
  \quad\text{for all $x \in \bar{B}(a,r)$.}
\]

\begin{lemma} 
\label{lemma:maximal:principle:etc}
Let $\Psi$ be an analytic function on $\bar{B}(a, r)$. 
\begin{parts}
\Part{(a)} \textup{[Maximum Principle]}
There is an $x^{\prime}
\in \bar{B}(a, r)$ such that 
\[
  |\Psi(x^{\prime})| = \Vert \Psi \Vert_{\bar{B}(a, r)}. 
\]
\Part{(b)}
For all $x,  y \in  \bar{B}(a, r)$, we have  
\[
  \bigl|\Psi(x) - \Psi(y)\bigr| \leq 
  \frac{\Vert \Psi \Vert_{\bar{B}(a, r)}}{r} \Vert x-y \Vert. 
\]
\end{parts}
\end{lemma}
\begin{proof}
We fix a~$b\in{K}^*$ with~$|b|=r$.
\par\noindent
(a)\enspace
For a proof when $\bar{B}(a, r)$ is the unit polydisk, i.e., $a=0$
and $r=1$, see \cite[\S~5.1.4, Propositions~3~and~4]{MR0746961}. As
in \cite[Proposition~1.1]{MR1794277}, the general case follows 
using the isomorphism
\[
  \bar{B}(a, r) \longrightarrow \bar{B}(0, 1), 
  \qquad
  x \longmapsto \frac{x-a}{b}.
\]
\par\noindent
(b)\enspace
To ease notation, we let~$I=(i_1,\ldots,i_N)$ and write~$(x-a)^I$ for
the product~$\prod_{j=1}^N (x_j-a_j)^{i_j}$.  Similarly for~$(y-a)^I$
and~$r^I=r^{i_1+\cdots+i_N}$.  
Then
\begin{align*}
  \bigl|\Psi(x) - \Psi(y)\bigr| 
  &= \biggl| \sum_I c_I \bigl( (x-a)^I - (y-a)^I \bigr) \biggr|\\
  &\le \sup_I |c_I|\cdot \bigl| (x-a)^I - (y-a)^I \bigr| \\
  &\le \bigl(\sup_I |c_I|r^I \bigr) \cdot 
       \sup_I\left| \left(\frac{x-a}{b}\right)^I - \left(\frac{y-a}{b}\right)^I
       \right| \\
  &= \|\Psi\|_{\bar{B}(a,r)} \cdot
       \sup_I\left| \left(\frac{x-a}{b}\right)^I - \left(\frac{y-a}{b}\right)^I
       \right|.
\end{align*}
We now use the fact that for all~$I$ and~$j$ there exist
polynomials $F_{I,j}(X,Y)\in\ZZ[X,Y]$ such that
\[
  X^I - Y^I 
  := \biggl(\prod_{j=1}^N X_j^{i_j}\biggr) - \biggl(\prod_{j=1}^N Y_j^{i_j}\biggr)
  = \sum_{j=1}^N (X_j-Y_j)F_{I,j}(X,Y).
\]
Putting~$X=(x-a)/b$ and~$Y=(y-a)/b$ and using the triangle inequality
yields
\[
  \left| \left(\frac{x-a}{b}\right)^I - \left(\frac{y-a}{b}\right)^I \right|
  \le \max_{1\le j\le N} \left|\frac{x_j-y_j}{b}\right|
       \cdot \left|F_{I,j}\left( \frac{x-a}{b}, \frac{y-a}{b} \right)\right|.
\]
We know that~$|b|=r$ and~$x,y\in\bar{B}(a,r)$, and also~$F_{I,j}$ has
integer coefficients, so $\bigl|F_{I,j}((x-a)/b,(y-a)/b)\bigr|\le 1$.
Hence
\[
  \bigl|\Psi(x) - \Psi(y)\bigr| 
  \le \|\Psi\|_{\bar{B}(a,r)} \cdot
    \max_{1\le j\le N} \left|\frac{x_j-y_j}{b}\right|
  = \|\Psi\|_{\bar{B}(a,r)} \cdot \frac{\|x-y\|}{r}.
\]
\end{proof}

\begin{lemma}
\label{lemma:upper:bound}
Let $\Family$ be a family of analytic functions on $\bar{B}(a, r)$.
Assume that there is a constant $C >0$ such that
\[
  \bigl|\Psi(x)\bigr| \leq C 
  \qquad\text{for all $x \in \bar{B}(a, r)$ and all $\Psi \in \Family$.}
\]
Then for all $x, y \in \bar{B}(a, r)$ and all $\Psi, \Lambda \in \Family$ 
we have 
\[
  \bigl|\Psi(x)\Lambda(y) - \Psi(y)\Lambda(x)\bigr| 
     \leq \frac{C^2}{r} \Vert x- y \Vert.
\]
\end{lemma}
\begin{proof}
By Lemma~\ref{lemma:maximal:principle:etc}(a), we have 
$\Vert \Psi \Vert_{\bar{B}(a, r)} \leq C$ for all $\Psi \in \Family$.
Then for any $x, y \in \bar{B}(a, r)$ we have 
{\allowdisplaybreaks
\begin{align}
 \bigl|\Psi(x)\Lambda(y) &- \Psi(y)\Lambda(x)\bigr| \notag\\
 &  = \bigl|\Lambda(y)(\Psi(x) - \Psi(y)) 
       -  \Psi(y)(\Lambda(x) - \Lambda(y))\bigr| \notag\\
 &  \leq \max\bigl\{|\Lambda(y)|\cdot|\Psi(x) - \Psi(y)|,\;  
  |\Psi(y)|\cdot |\Lambda(x) - \Lambda(y)|\bigr\} \notag\\
 &  \leq \max\bigl\{\Vert \Lambda \Vert_{\bar{B}(a, r)} |\Psi(x) - \Psi(y)|, \; 
   \Vert \Psi \Vert_{\bar{B}(a, r)} |\Lambda(x) - \Lambda(y)|\bigr\} \notag\\
 &  \leq \frac{\Vert \Psi \Vert_{\bar{B}(a, r)} \cdot 
 \Vert \Lambda \Vert_{\bar{B}(a, r)} }{r} \Vert x- y \Vert
 \qquad \text{from Lemma~\ref{lemma:maximal:principle:etc}(b)} \notag\\
 &  \leq \frac{C^2}{r} \Vert x- y \Vert.
  \tag*{\qedsymbol}
\end{align}
}
\renewcommand{\qedsymbol}{}
\end{proof}

\section{The Fatou and Julia sets}
\label{section:fatou:julia}

In this section we recall the definition of the Fatou and Julia sets
for a family of maps on a metric space. In our case, the metric space
is~$\PP^N(K)$ with the metric induced by the chordal distance
function~$\dist$.

\begin{definition}
Let~$U$ be an open subset of~$\PP^N(K)$.  A family of maps~$\Family$
from~$U$ to~$\PP^N(K)$ is {\em equicontinuous} at a point~$P\in U$ if for
every~$\e>0$ there is a~$\d>0$ such that
\[
  \psi\bigl(D_\d(P)\bigr)  \subset D_\e\bigl(\psi(P)\bigr)
  \qquad\text{for all $\psi\in\Family$.}
\]
We note that the open disks~$D_\d(P)$ and~$D_\e\bigl(\psi(P)\bigr)$
may be replaced by closed disks~$\bar{D}_\d(P)$
and~$\bar{D}_\e\bigl(\psi(P)\bigr)$ without affecting the definition.
The family~$\Family$ is {\em equicontinuous on~$U$} if it is 
equicontinuous at every~$P \in U$. 
\par
We note that in general, equicontinuity at a point $P$  
is not an open condition, 
since~$\d$ may depend on both~$\e$ and~$P$.  In particular, it is
weaker than the related property of~$\Family$ being \emph{uniformly
continuous on~$U$}, in which a single~$\d$ is required to work for
every~$P\in U$.
\par
The family~$\Family$ is called (\emph{locally}) \emph{uniformly
Lipschitz at~$P\in U$} if there exists a
constant~$C=C(\Family,P)$ and a radius~$r=r(\Family,P)$ such that
\[
  \D\bigl(\psi(Q), \psi(R)\bigr) \leq C \D(Q, R)
  \qquad\text{for all $Q,R\in \bar{D}_r(P)$ and all $\psi\in\Family$.}
\]
In other words,~$\Family$ is locally uniformly Lipschitz at~$P$ if each
map in~$\Family$ is Lipschitz in some neighborhood of~$P$ and further
there is a single Lipschitz constant that works for
every~$\psi\in\Family$. 
\end{definition}

If the family~$\Family$ is equicontinuous on each open subsets~$U_\a$ 
of~$\PP(K)$, then it is equicontinuous on the union~$\bigcup_{\a} U_\a$. 
Taking collections~$\{U_\a\}$ to be all open subsets
of~$\PP^N(K)$ on which~$\Family$ is equicontinuous, we are led to 
the following definition. 
\par
For convenience, we say that a map~$\f:\PP^N\to\PP^N$ is
equicontinuous if the family of iterates~$\{\f^n\}_{n\ge1}$ is
equicontinuous, and similarly~$\f$ is locally uniformly Lipschitz
if its iterates are.

\begin{definition}
Let~$\f:\PP^N_K\to\PP^N_K$ be a morphism.  The {\em Fatou set of
$\f$}, denoted~$\Fatou(\f)$, is the union of all open subsets
of~$\PP^N(K)$ on which~$\f$ is equicontinuous.  Equivalently, the Fatou
set~$\Fatou(\f)$ is the largest open set such that the
family~$\{\f^n\}_{n=1}^{\infty}$ is equicontinuous at every point
of~$\Fatou(\f)$.  
The {\em Julia set of~$\f$}, denoted~$\Julia(\f)$,
is the complement of $\Fatou(\f)$.  Thus by definition the Fatou set
is open and the Julia set is closed.
\end{definition}

\section{The Green function on the Fatou and Julia sets}
\label{section:behavior}

In this section we characterize the Fatou set of~$\f$ as
the set on which the (modified) Green function~$\mG_\F$ is 
locally constant. Along the way, we prove that~$\f$ is locally
uniformly Lipschitz on the Fatou set.


\begin{theorem}
\label{thm:locally:const:and:F}
Let~$\f:\PP_K^N\to\PP_K^N$ be a morphism of degree~$d\ge2$ as usual,
let~$\F$ be a lift of~$\f$,
let~$\mG_\F$ be the \textup(modified\textup) Green function defined
by~\eqref{eqn:def:G} and~\eqref{eqn:modGfun}, and let~$P\in\PP^N(K)$.  
Then the following are equivalent:
\begin{parts}
\Part{(a)}
The iterates of~$\f$ are equicontinuous at every point in some
neighborhood of~$P$, i.e., $P\in\Fatou(\f)$.
\Part{(b)}
The iterates of~$\f$ are locally uniformly Lipschitz at~$P$.
\Part{(c)}
The function~$\mG_\F$ is constant on a neighborhood of~$P$.
\end{parts}
\end{theorem}
\begin{proof}
It is clear that being locally uniformly Lipschitz at~$P$ is stronger
than being equicontinuous in a neighborhood of~$P$, so~(b)
implies~(a).
\par
Next we show that~(a) implies~(c), so we let~$P\in\Fatou(\f)$ and
let~$g_\F$ be the usual function
\[
  g_\F(Q) = \frac{1}{d}\log \Vert \F(y) \Vert - \log \Vert y \Vert
  \qquad\text{for $y\in\pi^{-1}(Q)$}
\]
as in Theorem~\ref{theorem:g:Lipschitz}. We take $\e = \frac12|\Res(\f)|$
in the definition of equicontinuity and find a~$\d=\d(\e,P)>0$ so that
\begin{multline*}
  \D(P, Q) \leq\d \quad\Longrightarrow\quad
  \D(\f^n(P), \f^n(Q)) \le \e < |\Res(\f)| \\
   \quad\text{for all $Q$ and all $n \geq 0$.}
\end{multline*}
It follows from Theorem~\ref{theorem:g:Lipschitz} that 
\[
  \D(P, Q) \leq\d \halfquad\Longrightarrow\halfquad
  g_\F(\f^n(P)) = g_\F(\f^n(Q))
  \quad\text{for all $Q$ and all $n \geq 0$.}
\]
Then the series representation of~$\mG_\F$
given in Proposition~\ref{prop:modgreenfuncprops}(c) implies that
\[
  \mG_\F(P) = \mG_\F(Q)
  \quad\text{for all  $Q\in\bar{D}_\d(P)$.}
\]
Hence~$\mG_\F$ is constant on~$\bar{D}_\d(P)$, which completes the
proof that~(a) implies~(c).
\par
It remains to show that~(c) implies~(b). So we assume that~$\mG_\F$ is
constant on~$\bar{D}_\d(P)$ and need to prove that the iterates
of~$\f$ are Lipschitz on~$\bar{D}_\d(P)$ with a uniform Lipschitz
constant.  We choose a minimal lift
\text{$\F:(K^{N+1})^*\to(K^{N+1})^*$} of~$\f$ and define
functions~$g_{\F,n}$ by
\[
  g_{\F,n}(Q)=\frac{1}{d^n}\log\bigl\|\F^n(y)\bigr\| - \log\|y\|
  \quad\text{for $Q\in\PP^N(K)$ and $y\in\pi^{-1}(Q)$.}
\]
Then as in the proof of Proposition~\ref{prop:modgreenfuncprops}, we can
use a telescoping sum to write
\begin{align*}
  \mG_\F(Q) - g_{\F,n}(Q)
   &= \lim_{k\to\infty} \frac{1}{d^k}\log\bigl\|\F^k(y)\bigr\|
      -  \frac{1}{d^n}\log\bigl\|\F^n(y)\bigr\| \\
   &= \sum_{k=n}^{\infty} \left( \frac{1}{d^{k+1}} \log\Vert \F^{k+1}(y)\Vert
       -\frac{1}{d^k} \log\Vert \F^k(y)\Vert\right).
\end{align*}
Then Proposition~\ref{prop:ResFxdleFlexd} gives the estimate
\begin{align}
  \label{eqn:mGQfnQ}
  \bigl|\mG_\F(Q) - g_{\F,n}(Q)\bigr|
  &\le \sum_{k\ge n} \frac{1}{d^k}\left| \frac{1}{d} \log\Vert \F^{k+1}(y)\Vert
       - \log\Vert \F^k(y)\Vert\right| \notag\\
  &\le \sum_{k\ge n} \frac{1}{d^{k+1}}  \log\bigl|\Resultant(\f)\bigr|^{-1}
  = \frac{C_1}{d^n},
\end{align}
where for convenience we
let~$C_1=\frac{1}{d-1}\log\bigl|\Resultant(\f)\bigr|^{-1}$. (In
particular, the constant~$C_1$ only depends on~$\f$.)
\par
Recall that we have fixed a point~$P\in\PP^N(K)$. It would be
convenient if we would find an element~$h\in K^*$
satisfying~$\log|h|=\mG_\F(P)$, but even if~$K=\CC_p$, we only
have~\text{$|\CC_p^*|=p^\QQ$}. However,~$\log(p^\QQ)$ is dense
in~$\RR$, so we can find a sequence of elements~$h_n\in K^*$
satisfying
\begin{equation}
  \label{eqn:mGPhn}
  \bigl| \mG_\F(P) - \log|h_n| \bigr| \le \frac{1}{d^n}
  \qquad\text{for all $n\ge 0$.}
\end{equation}
\par
Now let~$Q\in\bar{D}_\d(P)$ and choose lifts~$x\in\pi^{-1}(P)$
and~$y\in\pi^{-1}(Q)$. Note that~$\mG_\F(Q)=\mG_\F(P)$, since by
assumption~$\mG_\F$ is constant on~$\bar{D}_\d(P)$. This allows us to
estimate
\begin{align*}
  \biggl|\frac{1}{d^n} \log\bigl\|h_n^{-d^n}\F^n(y)\| &- \log\|y\| \biggr| \\
  &= \bigl|g_{\F,n}(Q) - \log|h_n| \bigr| \\
  &= \bigl|g_{\F,n}(Q) - \mG_\F(Q) + \mG_\F(P) - \log|h_n| \bigr| \\
  &\le \bigl|g_{\F,n}(Q) - \mG_\F(Q)\bigr| 
         + \bigl| \mG_\F(P) -\log|h_n| \bigr| \\
  &\le \frac{C_1+1}{d^n}
    \qquad\text{from \eqref{eqn:mGQfnQ} and \eqref{eqn:mGPhn}.}
\end{align*}
Hence if we define a new sequence of functions~$(\L_{\F,n})_{n\ge0}$ by
the formula
\[
  \L_{\F,n}(y) = h_n^{-d^n}\F^n(y)
\]
and a new constant $C_2=e^{C_1+1}$, then these new functions
satisfy
\begin{equation}
  \label{eqn:LnyC2}
  C_2^{-1} \le \frac{\bigl\|\L_{\F,n}(y)\bigr\|}{\|y\|^{d^n}} \le C_2
  \quad\text{for all $\pi(y)\in\bar{D}_\d(P)$ and $n\ge0$.}
\end{equation}
Notice that~$\L_{\F,n}$ is also a lift of~$\f^n$, since we have merely
multiplied~$\F^n$ by a constant.
\par
Reordering the coordinates if necessary and dividing by the largest
one, we may assume without loss of generality that~$x\in\pi^{-1}(P)$
satisfies~$x_0=1=\|x\|$.  Thus if we let~$a=(x_1,\ldots,x_N)$, then
Lemma~\ref{lemma:comparison:disks}(b) says that there is an isometric
isomorphism
\begin{equation}
  \label{eqn:isomisomsBadDdP}
  \s : \bar{B}(a,\d) \longrightarrow \bar{D}_\d(P),
  \qquad
  \s(b_1,\ldots,b_N)=(1:b_1:\dots:b_N).
\end{equation}
Let
\[
  \Psi_{\F,n}(b_1,\ldots,b_N) = \L_{\F,n}(1,b_1,\ldots,b_N)
\]
be the dehomogenization of~$\L_{\F,n}$. Then~\eqref{eqn:LnyC2} gives
\[
  C_2^{-1} \le \bigl\| \Psi_{\F,n}(b) \bigr\| \le C_2
  \qquad\text{for all $b\in\bar{B}(a,\d)$ and $n\ge0$.}
\]
\par
Write the coordinate functions of~$\Psi_{\F,n}$ as
$\Psi_{\F,n}=(\Psi_{n0},\ldots,\Psi_{nN})$ and consider the family of
functions
\[
  \bigl\{ \Psi_{ni} : 0\le i\le N~\text{and}~n\ge0 \bigr\}.
\]
Every function in this family satisfies $\bigl|\Psi_{ni}(b)\bigr|\le C_2$
for all $b\in\bar{B}(a,\d)$, so Lemma~\ref{lemma:upper:bound} tells 
us that
\begin{multline*}
  \bigl|\Psi_{ni}(b)\Psi_{nj}(b') - \Psi_{ni}(b')\Psi_{nj}(b) \bigr|
  \le \frac{C_2^2}{\d}\|b-b'\| \\
  \text{for all $b,b'\in\bar{B}(a,\d)$, all $0\le i,j\le N$, and all $n\ge0$.}
\end{multline*}
Combining this with the lower bound~$\|\Psi_{\F,n}(b)\|\ge C_2^{-1}$ 
yields
\begin{multline*}
  \frac{\bigl|\Psi_{ni}(b)\Psi_{nj}(b') - \Psi_{ni}(b')\Psi_{nj}(b) \bigr|}
      {\bigl\|\Psi_{\F,n}(b)\bigr\|\cdot \bigl\|\Psi_{\F,n}(b')\bigr\|}
  \le \frac{C_2^4}{\d}\|b-b'\| \\
  \text{for all $b,b'\in\bar{B}(a,\d)$, all $0\le i,j\le N$, and all $n\ge0$.}
\end{multline*}
Now we take the maximum over all $0\le i,j\le N$ and use the definition
of the chordal distance and the isometry~\eqref{eqn:isomisomsBadDdP}.
This gives
\begin{multline*}
  \dist\bigl(\s\bigl(\Psi_{\F,n}(b)\bigr),\s\bigl(\Psi_{\F,n}(b')\bigr)\bigr)
  \le \frac{C_2^4}{\d} \dist\bigl(\s(b),\s(b')\bigr) \\
  \text{for all $b,b'\in\bar{B}(a,\d)$ and all $n\ge0$.}
\end{multline*}
>From the definitions we have~$\s\bigl(\Psi_{\F,n}(b)\bigr)=\f^n(\s(b))$ and
similarly for~$b'$, so letting~$\s(b)=Q$ and~$\s(b')=R$, we have
proven that
\[
  \dist\bigl(\f^n(Q),\f^n(R)\bigr) \le \frac{C_2^4}{\d}\dist(Q,R)
  \quad\text{for all $Q,R\in\bar{D}_\d(P)$ and all $n\ge0$.}
\]
Hence the iterates of~$\f$ are uniformly Lipschitz on the disk~$\bar{D}_\d(P)$,
since the Lipschitz constant~$C_2^4/\d$ depends only on~$P$ and~$\f$.
\end{proof}

Theorem~\ref{thm:locally:const:and:F} has a number of useful
corollaries.  We note that it is possible to prove these corollaries
directly from the definition of the Fatou set, but the use of the
Green function simplifies and unifies the proofs. The first is
actually a restatement of part of
Theorem~\ref{thm:locally:const:and:F}, but we feel that it is
sufficiently important to merit the extra attention.  This is
particularly true because some authors define the nonarchimedean Fatou
set in terms of equicontinuity and others define it in terms of
uniform continuity. The following corollary shows that the two
definitions are equivalent, and indeed they are also equivalent to the
stronger locally uniformly Lipschitz property.

\begin{corollary}
Let~$\f:\PP_K^N\to\PP_K^N$ be a morphism of degree~$d\ge2$.
Then~$\{\f^n\}_{n\ge0}$ is locally uniformly Lipschitz on its Fatou
set~$\Fatou(\f)$.  In other words, for every~$P\in\Fatou(\f)$ there
exists a~$\d=\d(\f,P)>0$ and a constant~$C=C(\f,P)$ so that
\[
  \dist\bigl(\f^n(Q),\f^n(R)\bigr) \le C \dist(Q,R)
  \quad\text{for all $Q,R\in\bar{D}_\d(P)$ and all $n\ge0$.}
\]
\end{corollary}
\begin{proof}
This is the implication \text{(a) ${}\Longrightarrow{}$ (b)} in
Theorem~\ref{thm:locally:const:and:F}.
\end{proof}

The complete invariance of the Fatou and Julia sets is also an easy
corollary of Theorem~\ref{thm:locally:const:and:F} and the fact
that~$\f$ is an open mapping.

\begin{corollary}
The Fatou set $\Fatou(\f)$ and the Julia set $\Julia(\f)$ are
completely invariant under $\f$.
\end{corollary}
\begin{proof}
Since the Julia set is the complement of the Fatou set, it suffices to
prove the invariance of~$\Fatou(\f)$ under~$\f$ and~$\f^{-1}$.
\par
Let $P \in\f^{-1}\bigl(\Fatou(\f)\bigr)$.
Theorem~\ref{thm:locally:const:and:F}
says that the Green function~$\mG_\F$ is
constant on some disk $\bar{D}_\e\bigl(\f(P)\bigr)$.  Since $\f$ is
continuous, we can find a~$\d$ satisfying
\[
  0<\d<\bigl|\Resultant(\f)\bigr|
  \qquad\text{and}\qquad
  \f\bigl(\bar{D}_\d(P)\bigr) \subset \bar{D}_\e\bigl(\f(P)\bigr).
\]
Then by assumption,~$\mG_\F$ is constant on the
set~$\f\bigl(\bar{D}_\d(P)\bigr)$.  We claim that~$\mG_\F$ is constant
on~$\bar{D}_\d(P)$.
\par
Proposition~\ref{prop:modgreenfuncprops} tells us that
the Green function~$\mG_\F$ satisfies the transformation property
\begin{equation}
  \label{eqn:mGQ1dmGfQgQ}
  \mG_\F(Q) = \frac{1}{d}\mG_\F\bigl(\f(Q)\bigr) + g_\F(Q),
\end{equation}
where~$g_\F$ is the function defined in
Theorem~\ref{theorem:g:Lipschitz}.  And we know that the
function~$\mG_\F\circ\f$ is constant on~$\bar{D}_\d(P)$. But
Theorem~\ref{theorem:g:Lipschitz} says that~$g_\F$ is also
constant on that disk since we have
chosen~$\d<\bigl|\Resultant(\f)\bigr|$. This proves that~$\mG_\F$ is
constant in a neighborhood of~$P$, so
Theorem~\ref{thm:locally:const:and:F} tells us that~$P\in\Fatou(\f)$.
Hence $\f^{-1}\bigl(\Fatou(\f)\bigr) \subseteq \Fatou(\f)$.
\par
For the other direction, let $P \in \Fatou(\f)$.
Theorem~\ref{thm:locally:const:and:F} says that we can find a $0 < \d
< \bigl|\Resultant(\f)\bigr|$ such that $\mG_\F$ is constant on
$\bar{D}_\d(P)$.  Since $\f$ is an open mapping
(Corollary~\ref{cor:fisopen}), there is an $\e >0$ satisfying
\[
  \bar{D}_\e\bigl(\f(P)\bigr) \subset \f\bigl(\bar{D}_\d(P)\bigr).
\]
We claim that~$\mG_\F$ is constant on~$\bar{D}_\e\bigl(\f(P)\bigr)$.
\par
For any~$Q\in~\bar{D}_\e\bigl(\f(P)\bigr)$, we write~$Q=\f(R)$
with~$R\in\bar{D}_\d(P)$ and use
the transformation formula~\eqref{eqn:mGQ1dmGfQgQ} to compute
\[
  \mG_\F(Q) 
  = \mG_\F\bigl(\f(R)\bigr) = d\mG_\F(R) - dg_\F(R).
\]
The function~$\mG_\F$ is constant on~$\bar{D}_\d(P)$, and
since~$\d<\bigl|\Resultant(\f)\bigr|$,
Theorem~\ref{theorem:g:Lipschitz} tells us that~$g_\F$ is also constant
on~$\bar{D}_\d(P)$. Hence~$\mG_\F$ is constant
on~$\bar{D}_\e\bigl(\f(P)\bigr)$, so
Theorem~\ref{thm:locally:const:and:F} tells us
that~$\f(P)\in\Fatou(\f)$.  This completes the proof
that~$\f\bigl(\Fatou(\f)\bigr)\subset\Fatou(\f)$, which is the other
inclusion.
\end{proof}

\section{Good reduction and the Fatou set}
\label{section:goodredandfatou}

Roughly speaking, a morphism~$\f:\PP^N_K\to\PP^N_K$ has good
reduction at a point~$P\in\PP^N(K)$ if the reduction of~$\f$ to the
residue field~$k$ of~$K$ is well behaved at the reduction of~$P$.
In this section we show that~$\f$ has good reduction at~$P$ if and only
if~$\f$ is nonexpanding in a neighborhood of~$P$ (whose radius we specify
exactly). We then show how the locus of good reduction for~$\f$ can be
used to describe a subset of the Fatou set~$\Fatou(\f)$. We begin with some
definitions.

\begin{definition}
The morphism $\f$ has {\em good reduction} at $P$ if there is a lift
$\F$ of $\f$ and a lift $x$ of $P$ satisfying 
\begin{equation}
\label{eqn:good:reduction:at:P}
  \Vert x \Vert =1 \quad\text{and}\quad 
  \Vert \F \Vert= 1 \quad\text{and}\quad
  \Vert\F(x)\Vert = 1. 
\end{equation}
We write
\[
  U^{\good}(\f) = 
  \{ P \in \PP^N(K) \;:\; \text{$\f$ has good reduction at $P$}\}
\]
for the set of points at which~$\f$ has good reduction, and we
write $U^{\bad}(\f)$ for the complementary set where~$\f$ has bad
reduction.
\par
We say that~$\f$ has \emph{orbital good reduction} at~$P$ if
$\Ocal_\f(P)\subset U^{\good}(\f)$, i.e., if~$\f$ has good reduction
at every point in the forward orbit of~$P$. We denote the set of
such points by
\[
  U^{\orbital}(\f) =
  \{ P \in \PP^N(K) \;:\; \text{$\f$ has orbital good reduction at $P$}\}.
\]
Equivalently,
\[
  U^{\orbital}(\f) = 
   \smash[t]{\bigcap_{n=0}^{\infty} \f^{-n}(U^{\good}(\f)).}
\]
\end{definition}

\begin{remark}
Since any two lifts of~$\f$ differ by a constant, it is easy to see
that~$\f$ has good reduction at~$P$ if and only if every minimal
lift~$\F$ of~$\f$ and every lift~$x$ of~$P$ satisfying~$\|x\|=1$
also satisfies~$\bigl\|\F(x)\bigr\|=1$.
\end{remark}

\begin{remark}
It follows easily from Proposition~\ref{prop:ResFxdleFlexd} that if~$\f$
has (global) good reduction in the sense of
Remark~\ref{remark:good:reduction}, then $U^{\good}(\f)=\PP^N(K)$.
Conversely, if~$U^{\good}(\f)=\PP^N(K)$, then~$\|\F(x)\|=\|x\|^d$ for
all $x\in(K^{N+1})^*$, i.e.,~$g_\F$ is identically~$0$.
Proposition~\ref{prop:modgreenfuncprops} then implies that~$\mG_\F$ is
identically~$0$, and hence~\cite[Proposition~12]{KSEqualCanHt} tells
us that~$\f$ has (global) good reduction. In conclusion,
$|\Res(\f)|=1$ if and only if~$\f$ has good reduction at every point
of~$\PP^N(K)$.
\end{remark}

\begin{remark}
An alternative way to define $\f$ having good reduction at $P$ is the
existence of a lift~$\Phi$ with $K$-integral coefficients and an~$x$
with $K$-integral coordinates so that the image point $\Phi(x)$ has at
least one coordinate that is a $K$-unit.  This allows us to reduce
modulo the maximal ideal to obtain points defined over the residue
field $k$ of $K$, and we obtain the formulas
\[
  \tilde{\F}(\tilde{x}) 
  = \widetilde{\F(x)} \in \AA^{N+1}(k)\setminus\{0\}
  \qquad\text{and}\qquad
  \tilde{\f}(\tilde{P}) = \widetilde{\f(P)} \in \PP^N(k).
\]
\end{remark}

\begin{remark}
Our ad hoc definition of good reduction is convenient for
calculations, but we note that it is equivalent to the usual scheme
theoretic definition.  Thus let~$R$ be the ring of integers of~$K$ and
let~$k$ be the residue field. A point~$P\in\PP^N(K)$ induces a
section~\text{$s_P:\Spec(R)\to\PP^N_R$}, and we
write~$\tilde{P}=s_P\bigl(\Spec(k)\bigr)$ for the intersection of the
section with the special fiber~$\PP^N_k\subset\PP^N_R$.  Then~$\f$ has
good reduction at~$P$ if there is a rational
map~$\bar\f:\PP^N_R\to\PP^N_R$ whose restriction to the generic
fiber~$\PP^N_K$ is~$\f$ and such that~$\bar\f$ is defined
at~$\tilde{P}$.
\end{remark}

\begin{proposition}
\label{prop:goodredprops}
Let~$\f:\PP^N_K\to\PP^N_K$ be a morphism of degree~$d\ge2$,
let~$\F$ be a minimal lift of~$\f$,
and let~$P\in\PP^N(K)$. Consider the following five statements.
\begin{parts}
\Part{(a)}
$P\in U^{\good}(\f)$.
\vspace{1\jot}
\Part{(b)}
$D_{|\Resultant(\f)|}(P) \subset U^{\good}(\f)$.
\vspace{1\jot}
\Part{(c)}
$g_\F(P)=0$.
\vspace{1\jot}
\Part{(d)}
$\dist\bigl(\f(Q),\f(R)\bigr) \le \dist(Q,R)$
   for all $Q,R\in D_{|\Resultant(\f)|}(P)$.
\vspace{1\jot}
\Part{(e)}
$\f$ is nonexpanding in some neighborhood of~$P$.
\end{parts}
Then we have the following implications:
\begin{equation}
  \label{eqn:implications}
  \textup{(a)}
    \Longleftrightarrow
  \textup{(b)}
    \Longleftrightarrow
  \textup{(c)}
    \Longrightarrow
  \textup{(d)}
    \Longrightarrow
  \textup{(e)}
\end{equation}
In particular,~$U^{\good}(\f)$ is an open set.
\end{proposition}
\begin{proof}
It is clear that~(b) implies~(a) and~(d) implies~(e).  For the
remainder of this proof we fix a lift~$x$
of~$P$ satisfying~$\|x\|=1$.
\par
We first prove that~(a) implies~(b), so let~$P\in U^{\good}(\f)$.  The
good reduction condition \eqref{eqn:good:reduction:at:P} tells us
that~$\bigl\|\F(x)\bigr\|=1$.  Now let~$Q \in D_{|\Resultant(\f)|}(P)$
and choose a lift~$y$ of~$Q$ satisfying $\Vert y \Vert =1$. Then
Theorem~\ref{theorem:g:Lipschitz} tells us that~$\Vert \F(x) \Vert =
\Vert \F(y) \Vert$, so~$y$ also satisfies the good reduction
conditions \eqref{eqn:good:reduction:at:P}.  Hence $Q \in
U^{\good}(\f)$.
\par
We next prove that~(a) implies~(d), so let~$P\in U^{\good}(\f)$ and
let~$Q,R\in D_{|\Resultant(\f)|}(P)$. In particular, since we already
proved that~(a) implies~(b), we see that~$Q,R\in U^{\good}(\f)$.
Hence if we choose lifts~$y$ of~$Q$ and~$z$ of~$R$
satisfying~$\|y\|=\|z\|=1$, then the definition of good reduction
implies that
\[
  \bigl\|\F(y)\bigr\| = \bigl\|\F(z)\bigr\| = 1.
\]
Writing $\F = (\F_0,\ldots,\F_N)$, we proved earlier
(see~\eqref{eqn:FiFjFjFi} in the proof of
Theorem~\ref{theorem:f:Lipschitz}) that
\begin{equation}
  \left\vert \F_i(y)\F_j(z) - \F_j(y)\F_i(z)\right\vert 
  \leq \D(Q, R).
\end{equation}
Dividing by $\bigl\|\F(y)\bigr\| = \bigl\|\F(z)\bigr\| = 1$ yields
$\D\bigl(\f(Q), \f(R)\bigr) \leq \D(Q, R)$.
\par
It remains to show that~(a) and~(c) are equivalent. By definition, the
function~$g_\F$ is given by
\[
  g_\F(P) = \frac{1}{d}\log\bigl\|\F(x)\bigr\| - \log\|x\|.
\]
We have normalized~$x$ to satisfy~$\|x\|=1$ and by definition,~$\f$
has good reduction at~$P$ if and only if~$\bigl\|\F(x)\bigr\|=1$.
Hence~$P\in U^{\good}(\f)$ if and only if~$g_\F(P)=0$.
\par
This completes the proof of the implications~\eqref{eqn:implications}
Finally, it is clear from $\text{(a)}\Rightarrow\text{(b)}$
that~$U^{\good}(\f)$ is an open set.
\end{proof}


We conclude with a proposition describing the set of points of
orbital good reduction.

\begin{proposition}
\label{prop:G0orbitalFatou}
Let~$\f:\PP^N_K\to\PP^N_K$ be a morphism of degree~$d\ge2$ and
let~$\F$ be a minimal lift of~$\f$. 
\begin{parts}
\Part{(a)}
Let~$P\in U^{\orbital}(\f)$ and let~$Q,R\in D_{|\Resultant(\f)|}(P)$.
Then
\begin{equation}
  \label{eqn:DfnQfnR}
  \dist(\f^n(Q), \f^n(R)) \leq \dist(Q, R) 
  \quad\text{for all $n \geq1$}.
\end{equation}
\Part{(b)}
Let~$P\in U^{\orbital}(\f)$. Then
\begin{equation}
  \label{eqn:DRfPUorb}
  D_{|\Resultant(\f)|}(P) \subset U^{\orbital}(\f).
\end{equation}
In particular,~$U^{\orbital}(\f)$ is an open set.
\Part{(c)}
\leavevmode\vspace{-10pt}
\begin{equation}
  \label{eqn:UorfPGFP0}
  U^{\orbital}(\f)
  =\bigl\{P\in\PP^N(K) : \mG_\F(P)=0 \bigr\}
  \subseteq \Fatou(\f). 
\end{equation}
\Part{(d)}
$\mG_\F$ is strictly negative on~$U^{\bad}(\f)$.
\vspace{1\jot}
\Part{(e)}
The set $\bigl\{P\in\PP^N(K) : \mG_\F(P)=0\bigr\}$ is
an open set.
\end{parts}
\end{proposition}
\begin{proof}
(a,b)
Let~$P\in U^{\good}(\f)$.
We first use induction on~$n$ to prove~\eqref{eqn:DfnQfnR}
with~$R=P$. It is clearly true for~$n=1$.  Let~$Q\in
D_{|\Resultant(\f)|}(P)$ and assume that~\eqref{eqn:DfnQfnR}
with~$R=P$ is true for~$n$.  Then in particular we have
\[
  \dist(\f^n(Q), \f^n(P)) \leq \dist(Q,P) \leq \bigl|\Resultant(\f)\bigr|.
\]
Thus~$\f^n(Q)\in D_{|\Resultant(\f)|}\bigl(\f^n(P)\bigr)$,
and we know that~$\f^n(P)\in U^{\good}(\f)$, so
Proposition~\ref{prop:goodredprops}(d) tells us that
\[
  \dist\bigl(\f^{n+1}(Q),\f^{n+1}(P)\bigr)
  \le \dist\bigl(\f^{n}(Q),\f^{n}(P)\bigr).
\]
Then the induction hypothesis gives
$\dist\bigl(\f^{n+1}(Q),\f^{n+1}(P)\bigr)\le\dist(Q,P)$.
This proves that~\eqref{eqn:DfnQfnR} is true for~$R=P$.
\par
In particular, we have shown that if~$Q\in D_{|\Resultant(\f)|}(P)$,
then
\[
  \f^n(Q)\in D_{|\Resultant(\f)|}\bigl(\f^n(P)\bigr)
  \quad\text{for all $n\ge0$.}
\]
By assumption we have~$\f^n(P)\in U^{\good}$, so
Proposition~\ref{prop:goodredprops}(b) implies that~$\f^n(Q)\in
U^{\good}(\f)$. This holds for all~$n\ge0$, 
hence~\text{$Q\in U^{\orbital}(\f)$}, which proves the
inclusion~\eqref{eqn:DRfPUorb}.  And clearly~\eqref{eqn:DRfPUorb}
implies that~$U^{\orbital}(\f)$ is an open set.
\par
We now show that~(a) is true for all~$Q,R\in D_{|\Resultant(\f)|}(P)$.
From~(b) we have~$Q\in U^{\orbital}(\f)$. Further,
\[
  \dist(R,Q) \le \max\bigl\{ \dist(R,P),\dist(Q,P) \bigr\}
  \le |\Resultant(\f)|,
\]
so~$R\in D_{|\Resultant(\f)|}(Q)$. Hence we can apply
our preliminary version of~(a) to the point~$Q\in U^{\orbital}(\f)$
and the point~$R\in D_{|\Resultant(\f)|}(Q)$ to deduce that
\[
  \dist(\f^n(Q), \f^n(R)) \leq \dist(Q, R) 
  \quad\text{for all $n \geq1$}.
\]
\par\noindent(c)\enspace
We have
\begin{align*}
  \mG_\F(P)=0
  &\Longleftrightarrow
  g_\F\bigl(\f^n(P)\bigr)=0
    &&\text{for all~$n\ge0$ (Proposition~\ref{prop:modgreenfuncprops}),} \\
  &\Longleftrightarrow
  \f^n(P)\in U^{\good}(\f)
    &&\text{for all~$n\ge0$ (Theorem~\ref{prop:goodredprops}),} \\
  &\Longleftrightarrow
  P \in U^{\orbital}(\f).
\end{align*}
This proves the lefthand equality in~\eqref{eqn:UorfPGFP0}
\par
Next 
let~$P\in U^{\orbital}(\f)$. Then~(a) says that the iterates of~$\f$
are nonexpanding on the disk $D_{|\Resultant(\f)|}(P)$.  This is much
stronger than the assertion that~$\f$ is equicontinuous at every point
in the disk.  Hence~\text{$P\in\Fatou(\f)$}. This completes the
proof that \text{$U^{\orbital}(\f)\subset\Fatou(\f)$}.
\par\noindent(d)\enspace
From~(c) we see that
\[
  P\in U^{\bad}(\f)
  \Longleftrightarrow
  P\notin U^{\good}(\f)
  \Longrightarrow
  P\notin U^{\orbital}(\f)
  \Longrightarrow
  \mG_\F(P)\ne0.
\]
However, Proposition~\ref{prop:modgreenfuncprops} tells us that~$\mG_\F$ is
nonpositive, so~$\mG_\F(P)\ne0$ is equivalent to~$\mG_\F(P)<0$.
\par\noindent(e)\enspace
This is immediate from~(b) and~(c).
\end{proof}



\begin{thebibliography}{10}

\bibitem{MR1264116}
David~K. Arrowsmith and Franco Vivaldi.
\newblock Geometry of {$p$}-adic {S}iegel discs.
\newblock {\em Phys. D}, 71(1-2):222--236, 1994.

\bibitem{MR2244226}
Matthew Baker and Robert Rumely.
\newblock Equidistribution of small points, rational dynamics, and potential
  theory.
\newblock {\em Ann. Inst. Fourier}, 56(3):625--688, 2006.

\bibitem{MR1781331}
Robert~L. Benedetto.
\newblock {$p$}-adic dynamics and {S}ullivan's no wandering domains theorem.
\newblock {\em Compositio Math.}, 122(3):281--298, 2000.

\bibitem{MR1941304}
Robert~L. Benedetto.
\newblock Examples of wandering domains in {$p$}-adic polynomial dynamics.
\newblock {\em C. R. Math. Acad. Sci. Paris}, 335(7):615--620, 2002.

\bibitem{MR1981035}
Robert~L. Benedetto.
\newblock Non-{A}rchimedean holomorphic maps and the {A}hlfors {I}slands
  theorem.
\newblock {\em Amer. J. Math.}, 125(3):581--622, 2003.

\bibitem{MR1864626}
Jean-Paul B{\'e}zivin.
\newblock Sur les points p\'eriodiques des applications rationnelles en
  dynamique ultram\'etrique.
\newblock {\em Acta Arith.}, 100(1):63--74, 2001.

\bibitem{MR2031456}
Jean-Paul B{\'e}zivin.
\newblock Sur la compacit\'e des ensembles de {J}ulia des polyn\^omes
  {$p$}-adiques.
\newblock {\em Math. Z.}, 246(1-2):273--289, 2004.

\bibitem{MR0746961}
S.~Bosch, U.~G{\"u}ntzer, and R.~Remmert.
\newblock {\em Non-{A}rchimedean analysis}, volume 261 of {\em Grundlehren der
  Mathematischen Wissenschaften}.
\newblock Springer-Verlag, Berlin, 1984.
\newblock A systematic approach to rigid analytic geometry.

\bibitem{MR2244803}
A.~Chambert-Loir.
\newblock Mesures et \'equidistribution sur les espaces de Berkovich.
\newblock {\em J. Reine Angew. Math.}  595:215--235, 2006.

\bibitem{MR2137979}
Tien-Cuong Dinh and Nessim Sibony. 
\newblock Green currents for holomorphic automorphisms of compact K\"ahler  
   manifolds.
\newblock {\em J. Amer. Math. Soc.} 18(2):291--312, 2005. 

\bibitem{MR2221116}
Charles Favre and Juan Rivera-Letelier.
\newblock \'Equidistribution quantitative des points de petite hauteur sur la
 droite projective.
\newblock {\em Math. Ann.}  335(2):311--361, 2006. 

\bibitem{MR1387667}
Liang-Chung Hsia.
\newblock A weak {N}\'eron model with applications to {$p$}-adic dynamical
  systems.
\newblock {\em Compositio Math.}, 100(3):277--304, 1996.

\bibitem{MR1794277}
Liang-Chung Hsia.
\newblock Closure of periodic points over a non-{A}rchimedean field.
\newblock {\em J. London Math. Soc. (2)}, 62(3):685--700, 2000.

\bibitem{MR1142904}
J.-P. Jouanolou.
\newblock Le formalisme du r\'esultant.
\newblock {\em Adv. Math.}, 90(2):117--263, 1991.

\bibitem{KSEqualCanHt}
Shu Kawaguchi and Joseph~H. Silverman.
\newblock Dynamics of projective morphisms having identical canonical heights. 
\newblock {\em Proc. London Math. Soc. (3)}, to appear. 

\bibitem{MR1324210}
Patrick Morton and Joseph~H. Silverman.
\newblock Periodic points, multiplicities, and dynamical units.
\newblock {\em J. Reine Angew. Math.}, 461:81--122, 1995.

\bibitem{MR1769981}
Marcus Nilsson.
\newblock Cycles of monomial and perturbated monomial {$p$}-adic dynamical
  systems.
\newblock {\em Ann. Math. Blaise Pascal}, 7(1):37--63, 2000.

\bibitem{MR2040006}
Juan Rivera-Letelier.
\newblock Dynamique des fonctions rationnelles sur des corps locaux.
\newblock {\em Ast\'erisque}, (287):xv, 147--230, 2003.
\newblock Geometric methods in dynamics. II.

\bibitem{MR2018827}
Juan Rivera-Letelier.
\newblock Espace hyperbolique {$p$}-adique et dynamique des fonctions
  rationnelles.
\newblock {\em Compositio Math.}, 138(2):199--231, 2003.

\bibitem{MR2156656}
Juan Rivera-Letelier.
\newblock Wild recurrent critical points.
\newblock {\em J. London Math. Soc. (2)}, 72(2):305--326, 2005. 

\bibitem{MR1760844}
Nessim Sibony.
\newblock Dynamique des applications rationnelles de {$\mathbf{P}\sp k$}.
\newblock In {\em Dynamique et g\'eom\'etrie complexes (Lyon, 1997)}, volume~8
  of {\em Panor. Synth\`eses}, pages ix--x, xi--xii, 97--185. Soc. Math.
  France, Paris, 1999.

\bibitem{MR919501}
Joseph~H. Silverman.
\newblock Arithmetic distance functions and height functions in {D}iophantine
  geometry.
\newblock {\em Math. Ann.}, 279(2):193--216, 1987.

\bibitem{thuillier:thesis}
Amaury Thuillier.
\newblock {\em Th{\'e}orie du potentiel sur les courbes en g{\'e}om{\'e}trie
  analytique non archim{\'e}dienne. {A}pplications \`a la th{\'e}orie
  d'{A}rakelov}.
\newblock PhD thesis, Universit\'e Rennes, 2005.

\bibitem{MR1311351}
Shouwu Zhang.
\newblock Small points and adelic metrics
\newblock {\em J. Algebraic Geom.}, 4(2):281--300, 1995. 

\end{thebibliography}

\end{document}